\documentclass[draftclsnofoot, onecolumn, 12pt]{IEEEtran}

\usepackage{extarrows,paralist}
\usepackage{mhchem} 
\usepackage{siunitx} 
\usepackage{graphicx} 
\usepackage{amsmath,graphicx}
\usepackage{mathtools}

\usepackage{amsthm}
\newtheorem{lem}{Lemma}
\newtheorem{theorem}{Theorem}

\newtheorem{assump}{Assumption}

\usepackage[square, comma, sort&compress, numbers]{natbib}
\usepackage{mathtools}
\usepackage{graphicx}
\usepackage{float}
\usepackage{caption}
\usepackage{color}
\usepackage{placeins}
\usepackage{float}
\usepackage{tabularx,colortbl}
\usepackage[skip=2pt,font=footnotesize]{caption}
\usepackage{amsmath,subfigure}
\usepackage{listings}
\usepackage[outdir=./]{epstopdf}
\usepackage{amssymb}
\usepackage{multirow}
\usepackage{algorithm}
\usepackage{algpseudocode}
\usepackage{hyperref}

\usepackage{enumitem}
\allowdisplaybreaks
\usepackage[dvipsnames]{xcolor}

\makeatletter
\newcommand{\vast}{\bBigg@{4}}
\newcommand{\Vast}{\bBigg@{5}}
\makeatother

\def\ln{{\rm ln}}

\def\mc{\mathcal}
\def\mb{\mathbf}
\def\mbb{\mathbb}

\def\wh{\widehat}

\def\ul{\underline}
\def\bs{\boldsymbol}
\def\ol{\overline}
\newcommand{\mn}[1]{{\left\vert\kern-0.25ex\left\vert\kern-0.25ex\left\vert\kern0.3ex #1 
		\kern0.3ex\right\vert\kern-0.25ex\right\vert\kern-0.25ex\right\vert}}

\begin{document}
\title{\Large \textbf{FROST -- Fast row-stochastic optimization\\ with uncoordinated step-sizes}}
\author{Ran Xin,~\emph{Student Member,~IEEE}, Chenguang Xi,~\emph{Member,~IEEE},\\ and Usman A. Khan,~\emph{Senior Member,~IEEE}
\thanks{
The authors are with the ECE Department at Tufts University, Medford, MA; {\texttt{ran.xin@tufts.edu, khan@ece.tufts.edu}}. This work has been partially supported by an NSF Career Award \# CCF-1350264.}
}
\maketitle

\begin{abstract} 
In this paper, we discuss distributed optimization over directed graphs, where doubly-stochastic weights cannot be constructed. Most of the existing algorithms overcome this issue by applying push-sum consensus, which utilizes column-stochastic weights. {\color{black}The formulation of column-stochastic weights requires} each agent to know (at least) its out-degree, which may be impractical in e.g., broadcast-based communication protocols. In contrast, we describe FROST (Fast Row-stochastic-Optimization with uncoordinated STep-sizes), an optimization algorithm applicable to directed graphs that does not require the knowledge of out-degrees; the implementation of which is straightforward as each agent locally assigns weights to the incoming information and locally chooses a suitable step-size. We show that FROST converges linearly to the optimal solution for smooth and strongly-convex functions given that the largest step-size is positive and sufficiently small.
\end{abstract}


\section{Introduction}\label{s1}
In this paper, we study distributed optimization, where~$n$ agents are tasked to solve the following problem:  
$$
\min_{\mb{x}\in\mathbb{R}^n}F(\mb{x}) \triangleq \frac{1}{n}\sum_{i=1}^{n}f_i(\mb{x}),
$$
where each objective,~$f_i:\mathbb{R}^p\rightarrow\mathbb{R}$, is private and known only to agent~$i$. The goal of the agents is to find {\color{black}the global minimizer} of the aggregate cost,~$F(\mb{x})$, via local communication with their neighbors and without revealing their private objective functions. This formulation has recently received great attention due to its extensive applications in e.g., machine learning~\cite{forero2010consensus,distributed_Boyd,raja2016cloud,wai2018multi,di2013sparse}, control~\cite{jadbabaie2003coordination}, cognitive networks,~\cite{distributed_Mateos,distributed_Bazerque}, and source localization~\cite{distributed_Rabbit,safavi2018distributed}.

Early work on this topic includes Distributed Gradient Descent (DGD)~\cite{DOPT1,uc_Nedic}, which is computationally simple but is slow due to a diminishing step-size. The convergence rates are~$\mc{O}(\frac{\log k}{\sqrt{k}})$ for general convex functions and~$\mc{O}(\frac{\log k}{k})$ for strongly-convex functions, where~$k$ is the number of iterations. With a constant step-size, DGD converges faster albeit to an inexact solution~\cite{DGD_Yuan,balancing}. Related work also includes methods based on the Lagrangian dual~\cite{dual_Terelius,ADMM_Mota,ADMM_Wei,ADMM_Shi}to achieve faster convergence, albeit at the expense of more computation. To achieve both fast convergence and computational simplicity, some fast distributed first-order methods have been proposed. A Nesterov-type approach~\cite{fast_Gradient} achieves~$\mc{O}(\frac{\log k}{k^2})$ for smooth convex functions with bounded gradient assumption. EXTRA~\cite{EXTRA} exploits the difference of two consecutive DGD iterates to achieves a linear convergence to the optimal solution. Exact Diffusion~\cite{exactdiffusion1,exactdiffusion2} applies an Adapt-then-Combine structure~\cite{diffusion} to EXTRA and generalizes the symmetric doubly-stochastic weights required in EXTRA to locally-balanced row-stochastic weights over undirected graphs. {\color{black}Of significant relevance to this paper is a distributed gradient tracking technique built on dynamic consensus~\cite{DAC}, which enables each agent to asymptotically learn the gradient of the global objective function. This technique was first proposed simultaneously in~\cite{AugDGM,preNEXT}. Refs.~\cite{AugDGM,harness} combine it with the DGD structure to achieve improved convergence for smooth and convex problems. Refs.~\cite{preNEXT,NEXT}, on the other hand, propose the NEXT framework for a more general class of non-convex problems.}  

All of the aforementioned methods assume that the multi-agent network is undirected. In practice, it may not be possible to achieve undirected communication. It is of interest, thus, to develop algorithms that are fast and are applicable to arbitrary directed graphs. The challenge here lies in the fact that doubly-stochastic weights, standard in many distributed optimization algorithms, cannot be constructed over arbitrary directed graphs. In particular, the weight matrices in directed graphs can only be either row-stochastic or column-stochastic, but not both.

We now discuss related work on directed graphs. Early work based on DGD includes subgradient-push~\cite{opdirect_Tsianous,opdirect_Nedic} and Directed-Distributed Gradient Descent (D-DGD)~\cite{D-DGD,D-DPS}, with a sublinear convergence rate of~$\mc{O}(\frac{\log k}{\sqrt{k}})$. {\color{black}Some recent work extends these methods to asynchronous networks~\cite{AsySP,AsySPA,RstSP}}. To accelerate the convergence, DEXTRA~\cite{DEXTRA} combines push-sum~\cite{ac_directed0} and EXTRA~\cite{EXTRA} to achieve linear convergence given that the step-size lies in some non-trivial interval. This restriction on the step-size is later relaxed in ADD-OPT/Push-DIGing~\cite{xi2017add,diging}, which linearly converge for a sufficiently small step-size. Of relevance is also~\cite{sonata}, where distributed non-convex problems are considered with column-stochastic weights. More recent work~\cite{AB,ABM} proposes the~$\mc{AB}$ and~$\mc{AB}m$ algorithms, which employ both row- and uncoordinated- stochastic weights to achieve (accelerated) linear convergence over arbitrary strongly-connected graphs. Note that although the construction of doubly-stochastic weights is avoided, all of the aforementioned methods require each agent to know its out-degree to formulate {\color{black}doubly- or column-stochastic weights}. This requirement may be impractical in situations where the agents use a broadcast-based communication protocol. In contrast, Refs.~\cite{mai2016distributed,linear_row} provide algorithms that only use row-stochastic weights. Row-stochastic weight design is simple and is further applicable to broadcast-based methods. 

In this paper, we focus on optimization with row-stochastic weights following the recent work in~\cite{mai2016distributed,linear_row}. We propose a fast optimization algorithm, termed as FROST (Fast Row-stochastic Optimization with uncoordinated STep-sizes), which is applicable to both directed and undirected graphs with uncoordinated step-sizes among the agents. Distributed optimization (based on gradient tracking) with uncoordinated step-sizes has been previously studied in~\cite{AugDGM,xu2018convergence,nedic2017geometrically}, over undirected graphs with doubly-stochastic weights, and in~\cite{lu2018geometrical}, over directed graphs with column-stochastic weights. These works introduce a notion of heterogeneity among the step-sizes, defined respectively as the relative deviation of the step-sizes from their average in~\cite{xu2015augmented,xu2018convergence}, and as the ratio of the largest to the smallest step-size in~\cite{nedic2017geometrically,lu2018geometrical}. It is then shown that when the heterogeneity is small enough, i.e., the step-sizes are very close to each other, and when the largest step-size follows a bound as a function of the heterogeneity, the proposed algorithms linearly converge to the optimal solution. A challenge in this formulation is that choosing a sufficiently small, local step-size does not ensure small heterogeneity, while no step-size can be chosen to be zero. In contrast, a major contribution of this paper is that we establish linear convergence with uncoordinated step-sizes when the upper bound on the step-sizes is independent of any notion of heterogeneity. The implementation of FROST therefore is completely local, since each agent locally chooses a sufficiently small step-size, independent of other step-sizes, and locally assigns row-stochastic weights to the incoming information. In addition, our analysis shows that all step-sizes except one can be zero for the algorithm to work, which is a novel result in distributed optimization. We show that FROST converges linearly to the optimal solution for smooth and strongly-convex functions. 

\textbf{Notation:} We use lowercase bold letters to denote vectors and uppercase italic letters to denote matrices. The matrix,~$I_n$, represents the~$n\times n$ identity, whereas~$\mb{1}_n$ ($\mb{0}_n$) is the~$n$-dimensional uncoordinated vector of all~$1$'s ($0$'s). We further use~$\mb{e}_i$ to denote an $n$-dimensional vector of all~$0$'s except~$1$ at the~$i$th location. For an arbitrary vector,~$\mb{x}$, we denote its~$i$th element  by~$[\mb{x}]_i$ and $\mbox{diag}\{\mb{x}\}$ is a diagonal matrix with~$\mb{x}$ on its main diagonal. We denote by~$X\otimes Y$, the Kronecker product of two matrices,~$X$ and~$Y$. For a primitive, row-stochastic matrix,~$\ul{A}$, we denote its left and right Perron eigenvectors by $\bs{\pi}_r$ and~$\mb{1}_n$, respectively, such that~$\bs{\pi}_r^\top\mb{1}_n = 1$; similarly, for a primitive, column-stochastic matrix,~$\ul{B}$, we denote its left and right Perron eigenvectors by~$\mb{1}_n$ and~$\bs{\pi}_c$, respectively, such that~$\mb{1}_n^\top\bs{\pi}_c = 1$~\cite{matrix}. For a matrix,~$X$, we denote~$\rho(X)$ as its spectral radius and~$\mbox{diag}(X)$ as a diagonal matrix consisting of the corresponding diagonal elements of~$X$. The notation~$\|\cdot\|_2$ denotes the Euclidean norm of vectors and matrices, while~$\|\cdot\|_F$ denotes the Frobenius norm of matrices. {\color{black}Depending on the argument, we denote~$\|\cdot\|$ either as a particular matrix norm, the choice of which will be clear in Lemma~\ref{row_ctrac}, or a vector norm that is compatible with this matrix norm}, i.e.,~$\|X\mb{x}\|\leq\|X\|\|\mb{x}\|$ for all matrices,~$X$, and all vectors,~$\mb{x}$~\cite{matrix}.

We now describe the rest of the paper. Section~\ref{s2} states the problem and assumptions. Section~\ref{s3} reviews related algorithms {\color{black}that use doubly-stochastic or column-stochastic weights} and shows the intuition behind the analysis of these types of algorithms. In Section~\ref{s4}, we provide the main algorithm, FROST, proposed in this paper. In Section~\ref{s5}, we develop the convergence properties of FROST. Simulation results are provided in Section~\ref{s6} and Section~\ref{s7} concludes the paper.

\section{Problem Formulation}\label{s2}
{\color{black}Consider~$n$ agents communicating over a strongly-connected network},~$\mc{G}=(\mc{V},\mc{E})$, where~$\mc{V}=\{1,\cdots,n\}$ is the set of agents and~$\mc{E}$ is the set of edges,~$(i,j), i,j\in\mc{V}$, such that agent~$j$ can send information to agent~$i$, i.e.,~$j\rightarrow i$. Define~$\mc{N}_i^{{\scriptsize \mbox{in}}}$ as the collection of in-neighbors, i.e., the set of agents that can send information to agent~$i$. Similarly,~$\mc{N}_i^{{\scriptsize \mbox{out}}}$ as the set of out-neighbors of agent~$i$. Note that both~$\mc{N}_i^{{\scriptsize \mbox{in}}}$ and~$\mc{N}_i^{{\scriptsize \mbox{out}}}$ include agent~$i$. The agents are tasked to solve the following problem:
$$
\mbox{P1}:
	\quad\min_{\mb{x}} F(\mb{x})\triangleq\frac{1}{n}\sum_{i=1}^nf_i(\mb{x}),\nonumber
$$
where~$f_i:\mbb{R}^p\rightarrow\mbb{R}$ is a private cost function only known to agent~$i$. We denote the optimal solution of P1 as~$\mb{x}^*$. We will discuss different distributed algorithms related to this problem under the applicable set of assumptions, described below.
\begin{assump}\label{undig}The graph,~$\mc{G}$, is undirected and connected.
\end{assump}
\begin{assump}\label{dig}
The graph,~$\mc{G}$, is directed and strongly-connected.
\end{assump}
\begin{assump}\label{cvx_sub}
	Each local objective,~$f_i$, is convex with bounded subgradient.
\end{assump}
\begin{assump}\label{scvx}
	Each local objective,~$f_i$, is smooth and strongly-convex, i.e.,~$\forall i$ and~$\forall\mb{x}, \mb{y}\in\mbb{R}^p$,
		\begin{enumerate}[label={\roman*.}]
			\item there exists a positive constant~$l$ such that
$$
\qquad\|\mb{\nabla} f_i(\mb{x})-\mb{\nabla} f_i(\mb{y})\|_2\leq l\|\mb{x}-\mb{y}\|_2.
$$
	  \item there exists a positive constant~$\mu$ such that
	  $$
	  f_i(\mb{y})\geq f_i(\mb{x})+\nabla f_i(\mb{x})^\top(\mb{y}-\mb{x})+\frac{\mu}{2}\|\mb{x}-\mb{y}\|_2^2.
	  $$
	\end{enumerate}
Clearly, the Lipschitz-continuity and strong-convexity constants for the global objective function,~$F=\tfrac{1}{n}\sum_{i=1}^{n}f_i$, are~$l$ and~$\mu$, respectively. 
\end{assump}

\begin{assump}\label{iden}
	Each agent in the network has and knows its unique identifier, e.g.,~$1,\cdots,n$.

\noindent If this were not true, the agents may implement a finite-time distributed algorithm to assign such identifiers, e.g., with the help of task allocation algorithms,~\cite{kuhn55,saf_asil:14}, where the task at each agent is to pick a unique number from the set~$\{1,\ldots,n\}$. 
\end{assump}
\begin{assump}\label{outd}
	Each agent knows its out-degree in the network, i.e., the number of its out-neighbors.
\end{assump}
We note here that Assumptions~\ref{cvx_sub} and~\ref{scvx} do not hold together; when applicable,~the algorithms we discuss use either one of these assumptions but not both. We will discuss FROST, the algorithm proposed in this paper, under Assumptions~\ref{dig},~\ref{scvx},~\ref{iden}.

\section{Related work}\label{s3}
In this section, we discuss related distributed first-order methods and provide an intuitive explanation for each one of them.
\subsection{Algorithms using doubly-stochastic weights}
	A well-known solution to distributed optimization over undirected graphs is Distributed Gradient Descent (DGD)~\cite{DOPT1,uc_Nedic}, which combines distributed averaging with a local gradient step. Each agent~$i$ {\color{black}maintains a local estimate,~$\mb{x}_k^i$, of the optimal solution,~$\mb{x}^*$,} and implements the following iteration:
	\begin{equation} \label{DGD}
		\mb{x}_{k+1}^i=\sum_{j=1}^{n}w_{ij}\mb{x}_{k}^j-\alpha_k\nabla f_i\left(\mb{x}_{k}^i\right),
	\end{equation}
	{\color{black}where~$W=\{w_{ij}\}$ is doubly-stochastic and respects the graph topology}. The step-size~$\alpha_{k}$ is diminishing such that~$\sum_{k=0}^{\infty}\alpha_k=\infty$ and~$\sum_{k=0}^{\infty}\alpha_k^2<\infty$. Under the Assumptions~\ref{undig},~\ref{cvx_sub}, and~\ref{outd}, DGD converges to~$\mb{x}^*$ at the rate of~$\mc{O}(\frac{\log k}{\sqrt{k}})$. The convergence rate is slow because of the diminishing step-size. If a constant step-size is used in DGD, i.e.,~$\alpha_k=\alpha$, {\color{black}it converges faster to an error ball, proportional to~$\alpha$,} around~$\mb{x}^*$~\cite{DGD_Yuan,balancing}. {\color{black}This is because~$\mb{x}^*$ is not a fixed-point of the above iteration when the step-size is a constant.}
	
	To accelerate the convergence, Refs.~\cite{AugDGM,harness} recently propose a distributed first-order method based on gradient tracking, which uses a constant step-size and replaces the local gradient, at each agent in DGD, with an asymptotic estimator of the global gradient\footnote{EXTRA~\cite{EXTRA} is another related algorithm, which uses the difference between two consecutive DGD iterates to achieve linear convergence to the optimal solution.}.
	The algorithm is updated as follows~\cite{AugDGM,harness}:
	\begin{subequations}\label{harness_smooth}
		\begin{align}
		\mb{x}_{k+1}^i&=\sum_{j=1}^{n}w_{ij}\mb{x}_k^j-\alpha\mb{y}_k^i,\label{alg2a}\\
		\mb{y}_{k+1}^i&=\sum_{j=1}^{n}w_{ij}\mb{y}_k^{j}+\nabla f_i\left(\mb{x}_{k+1}^i\right)-\nabla f_i\left(\mb{x}_{k}^i\right),\label{alg2d}
		\end{align}
	\end{subequations}
	initialized with~$\mb{y}_0^i=\nabla f_i(\mb{x}_0^i)$ and an arbitrary~$\mb{x}_0^i$ at each agent. {\color{black}The first equation is essentially a descent method}, after mixing with neighboring information, where the descent direction is~$\mb{y}_k^i$, instead of~$\nabla f_i(\mb{x}_k^i)$ as was in Eq.~\eqref{DGD}. The second equation is a global gradient estimator when viewed as dynamic consensus~\cite{zhu2010discrete}, i.e.,~$\mb{y}_k^i$ asymptotically tracks the average of local gradients:~$\frac{1}{n}\sum_{i=1}^{n}\nabla f_i(\mb{x}_k^i)$. It is shown in Ref.~\cite{harness,xu2018convergence,diging} that~$\mb{x}_k^i$ converges linearly to~$\mb{x}^*$ under Assumptions~\ref{undig},~\ref{scvx},~\ref{outd}, with a sufficiently small step-size,~$\alpha$. Note that these methods, Eq.~\eqref{DGD} and Eqs.~\eqref{alg2a}-\eqref{alg2d}, are not applicable to directed graphs as they require doubly-stochastic weights. 
	
\subsection{Algorithms using column-stochastic weights}
We first consider the case when DGD in Eq.~\eqref{DGD} is applied to a directed graph and the weight matrix is column-stochastic but not row-stochastic. It can be obtained that~\cite{D-DGD}:
	\begin{equation}\label{DGD-C}
		\overline{\mb{x}}_{k+1}=\overline{\mb{x}}_k-\frac{\alpha_{k}}{n}\sum_{i=1}^{n}\nabla f_i(\mb{x}_k^i),
	\end{equation}
where~$\overline{\mb{x}}_k=\frac{1}{n}\sum_{i=1}^{n}\mb{x}_k^i$. From Eq.~\eqref{DGD-C}, {\color{black}it is clear that the average of the estimates,~$\overline{\mb{x}}_k$, converges to~$\mb{x}^*$, as Eq.~\eqref{DGD-C} can be viewed as a centralized gradient method if each local estimate~$\mb{x}_k^i$ converges to~$\overline{\mb{x}}_k$.} However, since the weight matrix is \textit{not row-stochastic}, the estimates of agents will not reach an agreement~\cite{D-DGD}. This discussion motivates combining DGD with an algorithm, called push-sum, briefly discussed next, that enables agreement over directed graphs with column-stochastic weights. 
	 
\subsubsection{Push-sum consensus}
Push-sum~\cite{ac_directed,ac_directed0} is a technique to achieve average-consensus over arbitrary digraphs. At time~$k$, each agent maintains two state vectors,~$\mb{x}_k^i$,~$\mb{z}_k^i\in\mathbb{R}^p$, and an auxiliary scalar variable,~$v_k^i$, initialized with~$v_0^i=1$. Push-sum performs the following iterations:
\begin{subequations}\label{Push-sum}
	\begin{align}
	v_{k+1}^i&=\sum_{j=1}^{n}b_{ij}v_k^j,\label{PSb}\\
	\mb{x}_{k+1}^i&=\sum_{j=1}^{n}b_{ij}\mb{x}_k^j\label{PSa}\\
	\mb{z}_{k+1}^i&=\frac{\mb{x}_{k+1}^i}{v_{k+1}^i},\label{PSc}
	\end{align}
\end{subequations}
where~$\ul{B}=\left\{b_{ij}\right\}$ is column-stochastic. Eq.~\eqref{PSb} can be viewed as an independent algorithm to asymptotically learn the right Perron eigenvector of~$\ul{B}$; recall that the right Perron eigenvector of~$\ul{B}$ is not~$\mb{1}_n$ because~$\ul{B}$ is not row-stochastic and we denote it by~$\bs{\pi}_c$. In fact, it can be verified that~$\lim_{k\rightarrow\infty}v_{i}(k)=n[\bs{\pi}_c]_i$ and that~$\lim_{k\rightarrow\infty}\mb{x}_{i}(k)=[\bs{\pi}_c]_i\sum_{i=1}^{n}\mb{x}_i(0)$. Therefore, the limit of~$\mb{z}_i(k)$, as the ratio of~$\mb{x}_i(k)$ over~$v_i(k)$, is the average of the initial values:
\begin{equation}
	\lim_{k\rightarrow\infty}\mb{z}_k^i = \lim_{k\rightarrow\infty}\frac{\mb{x}_k^i}{v_k^i}
	= \frac{[\bs{\pi}_c]_i\sum_{i=1}^{n}\mb{x}_i(0)}{n[\bs{\pi}_c]_i} = \frac{\sum_{i=1}^{n}\mb{x}_0^i}{n}. \nonumber
\end{equation}
In the next subsection, we present subgradient-push that applies push-sum to DGD, see~\cite{D-DGD,D-DPS} for an alternate approach that does not require eigenvector estimation of Eq.~\eqref{PSb}.

\subsubsection{Subgradient-Push}
To solve Problem P1 over arbitrary directed graphs, Refs.~\cite{opdirect_Tsianous,opdirect_Nedic} develop subgradient-push with the following iterations: 
\begin{subequations}\label{SP}
	\begin{align}
	v_{k+1}^i&=\sum_{j=1}^{n}b_{ij}v_k^j,\label{SPb}\\
	\mb{x}_{k+1}^i&=\sum_{j=1}^{n}b_{ij}\mb{x}_k^j-\alpha_{k} \nabla f_i\big(\mb{z}_k^i\big),\label{SPa}\\
	\mb{z}_{k+1}&=\frac{\mb{x}_{k+1}^i}{v_{k+1}^i},\label{SPc}
	\end{align}
\end{subequations}
initialized with~$v_0^i=1$ and an arbitrary~$\mb{x}_0^i$ at each agent. The step-size,~$\alpha_{k}$, satisfies the same conditions as in DGD. To understand these iterations, note that Eqs.~\eqref{SPb}-\eqref{SPc} are nearly the same as Eqs.~\eqref{PSb}-\eqref{PSc}, except that there is an additional gradient term in Eq.~\eqref{SPa}, which drives the limit of~$\mb{z}_k^i$ to~$\mb{x}^*$. Under the Assumptions~\ref{dig},~\ref{cvx_sub} and~\ref{outd}, subgradient-push converges to~$\mb{x}^*$ at the rate of~$\mc{O}(\frac{\log k}{\sqrt{k}})$. {\color{black}For extensions of subgradient-push to asynchronous networks, see recent work~\cite{AsySP,AsySPA,RstSP}}. We next describe an algorithm that significantly improves this convergence rate.

\subsubsection{ADD-OPT/Push-DIGing}\label{saddopt}
ADD-OPT~\cite{xi2017add}, extended to time-varying graphs in Push-DIGing~\cite{diging}, is a fast algorithm over directed graphs, which converges at a linear rate to~$\mb{x}^*$ under the Assumptions~\ref{dig},~\ref{scvx},  and~\ref{outd}, in contrast to the sublinear convergence of subgradient-push. The three vectors,~$\mb{x}_i(k)$,~$\mb{z}_i(k)$,~$\mb{y}_i(k)$, and a scalar~$v_i(k)$ maintained at each agent~$i$, are updated as follows:
\begin{subequations}\label{add-opt}
	\begin{align}
v_{k+1}^i&=\sum_{j=1}^{n}b_{ij}v_{k}^j,\label{addb}\\
	\mb{x}_{k+1}^i&=\sum_{j=1}^{n}b_{ij}\mb{x}_{k}^j-\alpha \mb{y}_k^i,\label{adda}\\
\mb{z}_{k+1}^i&=\frac{\mb{x}_{k+1}^i}{v_{k+1}^i},\label{addc} \\
	\mb{y}_{k+1}^i&=\sum_{j=1}^{n}b_{ij}\mb{y}_{k}^i+\nabla f_i\left(\mb{z}_{k+1}^i\right)-\nabla f_i\left(\mb{z}_{k}^i\right),\label{ADDOPTd}
	\end{align}
\end{subequations}
where each agent is initialized with~$v_0^i=1$,~$\mb{y}_0^i=\nabla f_i(\mb{x}_0^i)$, and an arbitrary~$\mb{x}_0^i$. We note here that ADD-OPT/Push-DIGing essentially applies push-sum to the algorithm in Eqs.~\eqref{alg2a}-\eqref{alg2d}, when the doubly-stochastic weights therein are replaced by column-stochastic weights. 

\subsubsection{The~$\mc{AB}$ algorithm}\label{slinear}
As we can see, subgradient-push and ADD-OPT/Push-DIGing, described before, have a nonlinear term that comes from the division by the eigenvector estimation. {\color{black}In contrast, the~$\mc{AB}$ algorithm, introduced in~\cite{AB} and extended to~$\mc{AB}m$ with the addition of a heavy-ball momentum term in~\cite{ABM} and to time-varying graphs in~\cite{tv-ab}}, removes this nonlinearity and remains applicable to directed graphs by a simultaneous application of row- and column-stochastic weights\footnote{See~\cite{D-DGD,D-DPS} for {\color{black}related work with sublinear rate based on surplus consensus~\cite{ac_Cai1}.}}. Each agent~$i$ maintains two variables:~$\mb{x}_k^i$,~$\mb{y}_k^i\in\mbb{R}^p$, where, as before,~$\mb{x}_k^i$ is the estimate of~$\mb{x}^*$, and~$\mb{y}_k^i$ tracks the average gradient,~$\frac{1}{n}\sum_{i=1}^{n}\nabla f_i(\mb{x}_k^i)$. The~$\mc{AB}$ algorithm, initialized with~$\mb{y}_0^i=\nabla f_i(\mb{x}_0^i)$ and arbitrary~$\mb{x}_0^i$ at each agent, performs the following iterations. 
\begin{subequations}\label{alg1}
	\begin{align}
	\mb{x}_{k+1}^i&=\sum_{j=1}^na_{ij}\mb{x}_k^j-\alpha\mb{y}_k^i,\label{alg1a}\\
	\mb{y}_{k+1}^i&=\sum_{j=1}^nb_{ij}\mb{y}_k^{j}+\nabla f_i\left(\mb{x}_{k+1}^i\right)-\nabla f_i\left(\mb{x}_k^i\right),\label{alg1d}
	\end{align}
\end{subequations}
where~$\ul{A}=\{a_{ij}\}$ is row-stochastic and~$\ul{B}=\{b_{ij}\}$ is column-stochastic. It is shown that~$\mc{AB}$ converges linearly to~$\mb{x}^*$ for sufficiently small step-sizes under the Assumptions~\ref{dig},~\ref{scvx} and~\ref{outd}~\cite{AB}. Therefore,~$\mc{AB}$ can be viewed as a generalization of the algorithm in Eqs.~\eqref{alg2a}-\eqref{alg2d} as the doubly-stochastic weights therein are replaced by row- and column-stochastic weights. Furthermore, it is shown in~\cite{ABM} that ADD-OPT/Push-DIGing in Eqs.~\eqref{addb}-\eqref{ADDOPTd} in fact can be derived from an equivalent form of~$\mc{AB}$ after a  state transformation on the~$\mb{x}_k$-update; see~\cite{ABM} for details. {\color{black}For applications of the~$\mc{AB}$ algorithm to distributed least squares, see, for instance,~\cite{abLS}.} 

\section{Algorithms using Row-stochastic Weights}\label{s4}
All of the aforementioned methods require at least each agent to know its out-degree in the network in order to construct doubly or column-stochastic weights. This requirement may be infeasible, e.g., when agents use broadcast-based communication protocols. Row-stochastic weights, on the other hand, are easier to implement in a distributed manner as every agent locally assigns an appropriate weight to each incoming variable from its in-neighbors. In the next section, we describe the main contribution of this paper, i.e., a fast optimization algorithm that uses only row-stochastic weights and uncoordinated step-sizes. 

To motivate the proposed algorithm, we first consider DGD in Eq.~\eqref{DGD} over directed graphs when the weight matrix in DGD is chosen to be row-stochastic, but not column-stochastic. From consensus arguments and the fact that the step-size~$\alpha_{k}$ goes to~$0$, it can be verified that the agents achieve agreement. However, this agreement is not on the optimal solution. This can be shown~\cite{D-DGD} by defining an accumulation state,~$\wh{\mb{x}}_k=\sum_{i=1}^{n}[\bs{\pi}_r]_i\mb{x}_k^i$, where~$\bs{\pi}_r$ is the left Perron eigenvector of the row-stochastic weight matrix, to obtain
\begin{equation} \label{DGD-R}
\widehat{\mb{x}}(k+1) = \wh{\mb{x}}(k) - \alpha_{k} \sum_{i=1}^{n} [\bs{\pi}_r]_i \nabla f_i\big(\mb{x}_i(k)\big). 
\end{equation}
It can be verified that the agents agree to the limit of the above iteration, which is suboptimal since this iteration minimizes a weighted sum of the objective functions and not the sum. This argument leads to a modification of Eq.~\eqref{DGD-R} that cancels the \textit{imbalance} in the gradient term caused by the fact that~$\bs{\pi}_r$ is not a vector of all~$1$'s, {\color{black}a consequence of losing the column-stochasticity in the weight matrix}. The modification, introduced in~\cite{mai2016distributed}, is implemented as follows:
\begin{subequations}\label{ROW}
	\begin{align}
	\mb{y}_{k+1}^i=&\sum_{j=1}^na_{ij}\mb{y}_k^j,\label{ROWb}\\
	\mb{x}_{k+1}^i=&\sum_{j=1}^na_{ij}\mb{x}_k^j-\alpha_{k}\frac{\nabla f_i\left(\mb{x}_{k}^i\right)}{[\mb{y}_k^i]_i},\label{ROWa}
	\end{align}
\end{subequations}
where~$\ul{A}=\{a_{ij}\}$ is row-stochastic and the algorithm is initialized with~$\mb{y}_0^i=\mb{e}_i$ and an arbitrary~$\mb{x}_0^i$ at each agent. Eq.~\eqref{ROWb} asymptotically learns the left Perron eigenvector of the row-stochastic weight matrix~$\ul{A}$, i.e.,~$	\lim_{k\rightarrow\infty}\mb{y}_k^i=\bs{\pi}_r,\forall i$. The above algorithm achieves a sublinear convergence rate of~$\mc{O}(\frac{\log k}{\sqrt{k}})$ under the Assumptions~\ref{dig},~\ref{cvx_sub}, and~\ref{iden}, see~\cite{mai2016distributed} for details.

\subsection{FROST (Fast Row-stochastic Optimization with uncoordinated STep-sizes)}
Based on the insights that gradient tracking and constant step-sizes provide exact and fast linear convergence, we now describe FROST that adds gradient tracking to the algorithm in Eqs.~\eqref{ROWb}-\eqref{ROWa} while keeping constant but uncoordinated step-sizes at the agents. Each agent~$i$ at the~$k$th iteration maintains three variables,~$\mb{x}_k^i,\mb{z}_k^i\in\mathbb{R}^p$, and~$\mb{y}_k^i\in\mathbb{R}^n$. At~$k+1$-th iteration, agent~$i$ performs the following update: 
\begin{subequations}\label{FROW}
	\begin{align}
	\mb{y}_{k+1}^i=&\sum_{j=1}^na_{ij}\mb{y}_{k}^j,\label{FROWb}\\
	\mb{x}_{k+1}^i=&\sum_{j=1}^na_{ij}\mb{x}_k^j-\alpha_i\mb{z}_{k}^j,\label{FROWa}\\
	\mb{z}_{k+1}^i=&\sum_{j=1}^na_{ij}\mb{z}_{k}^j+\frac{\nabla f_i\left(\mb{x}_{k+1}^i\right)}{[\mb{y}_{k+1}^i]_i}-\frac{\nabla f_i\left(\mb{x}_{k}^i\right)}{[\mb{y}_{k}^i]_i},\label{FROWc}
	\end{align}
\end{subequations}
where~$\alpha_i$'s are the uncoordinated step-sizes locally chosen at each agent and the row-stochastic weights,~$\ul{A}=\left\{a_{ij}\right\}$, respect the graph topology such that:
\begin{align*} 
a_{ij}&=\left\{
\begin{array}{rl}
>0,&j\in\mc{N}_i^{{\scriptsize \mbox{in}}},\\
0,&\mbox{otherwise},
\end{array}
\right.
\qquad
\sum_{j=1}^na_{ij}=1,~\forall i. 
\end{align*} 
The algorithm is initialized with an arbitrary~$\mb{x}_0^i$,~$\mb{y}_0^i=\mb{e}_i$, and~$\mb{z}_0^i=\nabla f_i(\mb{x}_0^i)$. We point out that the initial condition for Eq.~\eqref{FROWb} and the divisions in Eq.~\eqref{FROWc} require each agent to have a unique identifier. Clearly, Assumption~\ref{iden} is applicable here. Note that Eq.~\eqref{FROWc} is a modified gradient tracking update, first applied to optimization with row-stochastic weights in~\cite{linear_row}, where the divisions are used to eliminate the imbalance caused by the left Perron eigenvector of the (row-stochastic) weight matrix~$\ul{A}$. We note that the algorithm in~\cite{linear_row} requires identical step-sizes at the agents and thus is a special case of Eqs.~\eqref{FROWb}-\eqref{FROWc}. 

For analysis purposes, we write Eqs.~\eqref{FROWb}-\eqref{FROWc} in a compact vector-matrix form. To this aim, we introduce some notation as follows: let~$\mb{x}_k$,~$\mb{y}_k$, and~$\nabla\mb f(\mb{x}_k)$ collect the local variables~$\mb{x}_k^i$,~$\mb{y}_k^i$ and~$\nabla f_i\left(\mb{x}_k^i\right)$ in a vector in~$\mathbb{R}^{np}$, respectively, and define
\begin{eqnarray*}
\ul{Y}_k&=&[\mb{y}_k^1,\cdots,\mb{y}_k^n]^\top,\nonumber\\
Y_k &=& \ul{Y}_k \otimes I_p, \\
\widetilde{Y}_k&=&\mbox{diag}\big(Y_k\big),\nonumber\\
A &=& \ul{A} \otimes I_p, \\
\bs{\alpha} &=& [\alpha_1,\cdots,\alpha_n]^\top, \\
D &=& \mbox{diag}\{\bs{\alpha}\} \otimes I_p.
\end{eqnarray*}
Since the weight matrix~$\ul{A}$ is {\color{black}primitive with positive diagonals}, it is straightforward to verify that~$\widetilde{Y}_k$ is invertible for any~$k$. Based on the notation above, Eqs.~\eqref{FROWb}-\eqref{FROWc} can be written compactly as follows:
\begin{subequations}\label{alg1_matrix}
	\begin{eqnarray}
	\ul{Y}_{k+1}&=&\ul{A}~\ul{Y}_k,\label{alg1_mb}\\
	\mb{x}_{k+1}&=&A\mb{x}_k-D\mb{z}_k,\label{alg1_ma}\\
	\mb{z}_{k+1}&=&A\mb{z}_{k}+\widetilde{Y}_{k+1}^{-1} \nabla\mb{f}(\mb{x}_{k+1})-\widetilde{Y}_k^{-1} \nabla\mb{f}(\mb{x}_k),\label{alg1_md}
	\end{eqnarray}
\end{subequations}
where~$\ul{Y}_0=I_n,$~$\mb{z}_0=\nabla\mb{f}_0$, and~$\mb{x}_0$ is arbitrary. We emphasize that the implementation of FROST needs no knowledge of agent's out-degree anywhere in the network in contrast to the earlier related work in~\cite{opdirect_Nedic,opdirect_Tsianous,D-DGD,D-DPS,DEXTRA,xi2017add,diging,AB,ABM}. Note that Refs.~\cite{exactdiffusion1,exactdiffusion2} also use row-stochastic weights but require an additional locally-balanced assumption and are only applicable to undirected graphs. 

\section{Convergence Analysis}\label{s5}
In this section, we present the convergence analysis of FROST described in Eqs.~\eqref{alg1_mb}-\eqref{alg1_md}. We first define a few additional variables as follows:
\begin{eqnarray}
Y_\infty&=&\lim_{k\rightarrow\infty}Y_k, \nonumber\\
\widetilde{Y}_\infty&=&\mbox{diag}\big(Y_\infty\big), \nonumber\\
\nabla\mb{f}(\mb{x}^*)&=&[\nabla f_1(\mb{x}^*)^\top,\cdots,\nabla f_n(\mb{x}^*)^\top]^\top,\nonumber\\
\tau &=& \left\|A-I_{np}\right\|_2, \nonumber\\
\epsilon &=& \left\|I_{np}-Y_\infty\right\|_2, \nonumber \\
\overline{\alpha} &=& \max_i\{\alpha_i\}, \nonumber\\
y&=&\sup_k\left\|Y_k\right\|_2,\nonumber\\
\widetilde{y}&=&\sup_k\left\|\widetilde{Y}_k^{-1}\right\|_2.\nonumber
\end{eqnarray}
Since~$\ul{A}$ is primitive and row-stochastic, from the Perron-Frobenius theorem~\cite{matrix}, we note that~$Y_\infty =  \left(\mb{1}_n\bs{\pi}_r^\top\right)\otimes I_p$, where~$\bs{\pi}_r^\top$ is the left Perron eigenvector of~$\ul{A}$.

\subsection{Auxiliary relations}
We now start the convergence analysis with a key lemma regarding the contraction of the augmented weight matrix~$A$ under an arbitrary norm. 

\begin{lem}\label{row_ctrac}
	Let Assumption \ref{dig} hold and consider the augmented weight matrix~$A=\ul{A}\otimes I_p$. There exists a vector norm,~$\|\cdot\|$, such that~$\forall\mb{a}\in\mbb{R}^{np}$,
	\begin{align*}
	\left\|A\mb{a}-Y_\infty\mb{a}\right\|&\leq\sigma\left\|\mb{a}-Y_\infty\mb{a}\right\|, 
	\end{align*}
	where~$0<\sigma<1$ is some constant.
\end{lem}
\begin{proof}
	It can be verified that~$AY_{\infty}=Y_{\infty}$ and~$Y_{\infty}Y_{\infty}=Y_{\infty}$, which leads to the following relation:
	\begin{eqnarray}
	A\mb{a}-Y_\infty \mb{a}=(A-Y_{\infty})(\mb{a}-Y_{\infty}\mb{a}) \nonumber.
	\end{eqnarray}
	Next, from the Perron-Frobenius theorem, we note that~\cite{matrix}
	\begin{equation}
	\rho(A-Y_{\infty})=\rho\left(\ul{A}-\mb{1}_n\bs{\pi}_r^\top\right)<1; \nonumber
	\end{equation}
	thus there exists a matrix norm,~$\|\cdot\|$, with~$\left\|A-Y_{\infty}\right\| <1$ and a compatible vector norm,~$\|\cdot\|$, see Ch. 5 in~\cite{matrix}, such that
	\begin{eqnarray}
	\left\|A\mb{a}-Y_{\infty}\mb{a}\right\|
	\leq \mn{A-Y_{\infty}} \left\|\mb{a}-Y_{\infty}\mb{a}\right\|, \nonumber
	\end{eqnarray}
	and the lemma follows with~$\sigma=\left\|A-Y_{\infty}\right\|$. 
\end{proof}

As shown above, the existence of a norm in which the consensus process with row-stochastic matrix~$\ul{A}$ is a contraction does not follow the standard $2$-norm argument for doubly-stochastic matrices~\cite{harness,diging}. The ensuing arguments built on this notion of contraction under arbitrary norms were first introduced in~\cite{xi2017add} for column-stochastic weights and in~\cite{linear_row} for row-stochastic weights; these arguments are harmonized later to hold simultaneously for both row- and column-stochastic weights in~\cite{AB,ABM}. The next lemma, a direct consequence of the contraction introduced in Lemma~\ref{row_ctrac}, is a standard result from consensus and Markov chain theory~\cite{hornjohnson:13}. 
\begin{lem}\label{lem2}
	 Consider~$Y_k$, generated from the weight matrix~$\ul{A}$. We have:
	\begin{align}\label{DkDinfty1}
\begin{color}{black}	\left\|Y_k-Y_\infty\right\|_2\leq r\sigma^{k},\qquad  \forall k, \nonumber
\end{color}	\end{align}
where~$r$ is some positive constant and~$\sigma$ is the contraction factor defined in Lemma~\ref{row_ctrac}.
\end{lem}
\begin{proof}
Note that~$Y_k = \ul{A}^k\otimes I_p=A^k$ from~Eq.~\eqref{alg1_mb}, and
\begin{equation}
	Y_k-Y_\infty ~=~ A^k - Y_\infty ~=~ (A-Y_\infty)(A^{k-1} - Y_\infty)~=~ (A-Y_\infty)^k. \nonumber
\end{equation}
From Lemma~\ref{row_ctrac}, we have that
\begin{equation}
	\left\|Y_k-Y_\infty\right\| = \left\|(A-Y_\infty)^k\right\| \leq \sigma^k. \nonumber
\end{equation}
The proof follows from the fact that all matrix norms are equivalent.
\end{proof}
As a consequence of Lemma~\ref{lem2}, we next establish the linear convergence of the sequences~$\left\{ \widetilde{Y}_k^{-1}\right\}$ and~$\left\{ \widetilde{Y}_{k+1}^{-1}-\widetilde{Y}_k^{-1}\right\}$.
\begin{lem}\label{yy-}
	The following inequalities hold~$\forall k$:
	\begin{inparaenum}[(a)]
		\item
		$\left\|\widetilde{Y}_k^{-1}-\widetilde{Y}_{\infty}^{-1}\right\|_2\leq \sqrt{n}r\widetilde{y}^2\sigma^{k}$;
		\item 
		$\left\|\widetilde{Y}_{k+1}^{-1}-\widetilde{Y}_k^{-1}\right\|_2\leq 
		2\sqrt{n}r\widetilde{y}^2\sigma^{k}$.
	\end{inparaenum}
\end{lem}
\begin{proof}
	The proof of (a) is as follows:
	\begin{align}
	\left\|\widetilde{Y}_k^{-1}-\widetilde{Y}_{\infty}^{-1}\right\|_2
	&=\left\|\widetilde{Y}_k^{-1}(\widetilde{Y}_{\infty}-\widetilde{Y}_k)\widetilde{Y}_{\infty}^{-1}\right\|_2,\nonumber\\
	&\leq\left\|\widetilde{Y}_k^{-1}\right\|_2\left\|\widetilde{Y}_k-\widetilde{Y}_{\infty}\right\|_2\left\|\widetilde{Y}_{\infty}^{-1}\right\|_2,\nonumber\\
	&\leq \widetilde{y}^2\left\|\mbox{diag}\left(Y_k-Y_{\infty}\right)\right\|_2 \nonumber\\
	&\leq \sqrt{n}r\widetilde{y}^2\sigma^{k},\nonumber
	\end{align}
	where the last inequality uses Lemma~\ref{lem2} and the fact that~$\|X\|_F\leq \sqrt{n}\|{X}_2\|,\forall X\in\mathbb{R}^{n\times n}$. The result in~(b) is straightforward by applying~(a), i.e.,
	\begin{align*}
	\left\|\widetilde{Y}_{k+1}^{-1}-\widetilde{Y}_k^{-1}\right\|_2\leq&
	\left\|\widetilde{Y}_{k+1}^{-1}-\widetilde{Y}_{\infty}^{-1}\right\|_2+
	\left\|\widetilde{Y}_{\infty}^{-1}-\widetilde{Y}_k^{-1}\right\|_2, \\
	\leq&\sqrt{n}r\widetilde{y}^2\sigma^{k+1}+\sqrt{n}r\widetilde{y}^2\sigma^{k},
	\end{align*}
	which completes the proof.
\end{proof}
The next lemma presents the dynamics that govern the evolution of the weighted sum of~$\mb{z}_k$; recall that~$\mb{z}_k$, in Eq.~\eqref{alg1_md}, asymptotically tracks the average of local gradients,~$\frac{1}{n}\sum_{i=1}^{n}\nabla f_i\left(\mb{x}_k^i\right)$.
\begin{lem}\label{z}
	The following equation holds for all~$k$:
	\begin{equation}
	Y_{\infty}\mb{z}_k=Y_\infty\widetilde{Y}_k^{-1}\nabla\mb{f}(\mb{x}_k).
	\end{equation}
\end{lem}
\begin{proof}
	Recall that~$Y_\infty A=Y_\infty$. We obtain from Eq. \eqref{alg1_md} that
	\begin{align}
	Y_{\infty}\mb{z}_k&=Y_{\infty}\mb{z}_{k-1}+Y_\infty\widetilde{Y}_k^{-1}\nabla\mb{f}(\mb{x}_k)-Y_\infty\widetilde{Y}_{k-1}^{-1}\nabla\mb{f}(\mb{x}_{k-1}).\nonumber
	\end{align}
	Doing this iteratively, we have that
	\begin{align}
	Y_{\infty}\mb{z}_k=Y_{\infty}\mb{z}_0+Y_\infty\widetilde{Y}_k^{-1}\nabla\mb{f}(\mb{x}_k)-Y_\infty\widetilde{Y}_0^{-1}\nabla\mb{f}(\mb{x}_0).\nonumber
	\end{align}
With the initial conditions that~$\mb{z}_0=\nabla\mb{f}(\mb{x}_0)$ and~$\widetilde{Y}_0=I_{np}$, we complete the proof.
\end{proof}

The next lemma, a standard result in convex optimization theory from~\cite{bertsekas1999nonlinear}, states that the distance to the optimal solution contracts in each step in the centralized gradient method.
\vspace{-0.2cm}
\begin{lem}\label{centr_d}
	Let~$\mu$ and~$l$ be the strong-convexity and Lipschitz-continuity constants for the global objective function,~$F(\mb{x})$, respectively. Then~$\forall \mb{x}\in\mbb{R}^p$ and~$0<\alpha<\frac{2}{l}$, we have ~$$\left\|\mb{x}-\alpha\nabla F(\mb{x})-\mb{x}^*\right\|_2\leq\sigma_F\left\|\mb{x}-\mb{x}^*\right\|_2,$$ where~$\sigma_F=\max\left(\left|1-\alpha \mu\right|,\left|1-\alpha l \right|\right)$.
\end{lem}
With the help of the previous lemmas, we are ready to derive a crucial contraction relationship in the proposed algorithm. 

\subsection{Contraction relationship}
Our strategy to show convergence is to bound~$\|\mb{x}_{k+1}-Y_\infty\mb{x}_{k+1}\|$,~$\|Y_\infty\mb{x}_{k+1}-\mb{1}_n \otimes \mb{x}^*\|_2$, and~$\|\mb{z}_{k+1}-Y_\infty\mb{z}_{k+1}\|$ as a linear function of their values in the last iteration and~$\nabla\mb{f}(\mb{x}_k)$; this approach extends the work in~\cite{harness} on doubly-stochastic weights to row-stochastic weights. We will present this relationship in the next lemmas. Before we proceed, we note that since all vector norms are equivalent in~$\mathbb{R}^{np}$, there exist positive constants~$c,d$ such that:
$\|\cdot\|_2\leq c\|\cdot\|,\|\cdot\|\leq d\|\cdot\|_2.$
First, we derive a bound for~$\|\mb{x}_{k+1}-Y_\infty\mb{x}_{k+1}\|$, the consensus error of the agents.
\begin{lem} \label{1p}
	The following inequality holds,~$\forall k$:
	\begin{align}
	\|\mb{x}&_{k+1}-Y_\infty\mb{x}_{k+1}\|\leq\sigma\|\mb{x}_k-Y_\infty\mb{x}_k\|+\ol{\alpha} d\epsilon\|\mb{z}_k\|_2,
	\end{align}
	where~$d$ is the equivalence-norm constant such that~$\|\cdot\|\leq d\|\cdot\|_2$ and~$\ol{\alpha}$ is the largest step-size among the agents. 
\end{lem}
\begin{proof}
	Note that~$Y_\infty A=Y_\infty$. Using Eq.~\eqref{alg1_ma} and Lemma~\ref{row_ctrac}, we have
	\begin{align*}
	\|\mb{x}&_{k+1}-Y_\infty\mb{x}_{k+1}\| \nonumber\\
	=&~\left\|A\mb{x}_k-D\mb{z}_k-Y_{\infty}\left(A\mb{x}_k-D \mb{z}_k\right)\right\|\leq\sigma\|\mb{x}_k-Y_\infty\mb{x}_k\|+\ol{\alpha} d\epsilon\|\mb{z}_k\|_2,
	\end{align*}
	which completes the proof.
\end{proof}
Next, we derive a bound for~$\|Y_\infty\mb{x}_{k+1}-\mb{1}_n \otimes \mb{x}^*\|_2$, i.e., the optimality gap between the accumulation state of the network,~$Y_\infty\mb{x}_{k+1}$, and the optimal solution,~$\mb{1}_n \otimes \mb{x}^*$.
\begin{lem} \label{2p}
	If~$\bs{\pi}_r^\top\bs{\alpha}<\frac{2}{nl}$, the following inequality holds,~$\forall k$:
	\begin{align}\label{20}
	\|Y&_\infty\mb{x}_{k+1}-\mb{1}_n \otimes \mb{x}^*\|_2\nonumber\\
	\leq&~\overline{\alpha} \begin{color}{black}n\end{color}l c\|\mb{x}_k-Y_\infty\mb{x}_k\|
	+~\lambda\|Y_\infty\mb{x}_k-\mb{1}_n \otimes \mb{x}^*\|_2 \nonumber\\
	&+\overline{\alpha}yc\|\mb{z}_k-Y_\infty\mb{z}_k\| 
	+\overline{\alpha}\sqrt{n}ry\widetilde{y}^2\sigma^{k}\left\|\nabla\mb{f}(\mb{x}_k)\right\|_2, 
	\end{align}
	where~$\lambda=\max\left(\left|1-n\bs{\pi}_r^\top\bs{\alpha}\mu\right|,\left|1-n\bs{\pi}_r^\top\bs{\alpha}l \right|\right)$
	and~$c$ is the equivalence-norm constant such that~$\|\cdot\|_2\leq c\|\cdot\|$.
\end{lem}
\begin{proof}
	Recalling that $Y_{\infty}=(\mb{1}_n \bs{\pi}_r^\top)\otimes I_p$ and~$Y_\infty A=Y_\infty$, 
	We have the following:
	\begin{align}\label{21}
	\|Y&_\infty\mb{x}_{k+1}-\mb{1}_n \otimes \mb{x}^*\|_2 \nonumber\\
	=&~ \left\|Y_{\infty}\Big(A\mb{x}_k-D \mb{z}_k+(D-D)Y_\infty\mb{z}_k\Big)-\mb{1}_n \otimes \mb{x}^*\right\|_2, \nonumber\\
	\leq&
	~\left\|Y_\infty\mb{x}_k-Y_\infty DY_\infty\mb{z}_k-\mb{1}_n \otimes \mb{x}^*\right\|_2+\overline{\alpha}yc\|\mb{z}_k-Y_\infty\mb{z}_k\|.
	\end{align}
	Since the last term in the inequality above matches the second last term in Eq.~\eqref{20}, we only need to handle the first term. We further note that:
	\begin{align}
	Y_\infty DY_\infty &= \Big((\mb{1}_n\bs{\pi}_r^\top)\otimes I_p\Big)\Big(\mbox{diag}\{\bs{\alpha}\} \otimes I_p\Big)\Big((\mb{1}_n\bs{\pi}_r^\top)\otimes I_p\Big)=(\bs{\pi}_r^\top\bs{\alpha})Y_\infty. \nonumber
	\end{align}
	Now, we derive a upper bound for the first term in Eq.~\eqref{21},
	\begin{align}
	\|Y&_\infty\mb{x}_k-Y_\infty DY_\infty\mb{z}_k-\mb{1}_n \otimes \mb{x}^*\|_2\nonumber\\
	\leq&~ \left\|(\mb{1}_n \otimes I_p)\Big((\bs{\pi}_r^\top\otimes I_p)\mb{x}_k-\mb{x}^*-n(\bs{\pi}_r^\top\bs{\alpha})\nabla F\big((\bs{\pi}_r^\top\otimes I_p)\mb{x}_k\big)\Big)\right\|_2 \nonumber\\
	&+~\left\|n(\bs{\pi}_r^\top\bs{\alpha})(\mb{1}_n \otimes I_p)\nabla F\big((\bs{\pi}_r^\top\otimes I_p)\mb{x}_k\big)-(\bs{\pi}_r^\top\bs{\alpha})Y_\infty\mb{z}_k\right\|_2,\nonumber\\
	:=& ~s_1 +  s_2.
	\end{align}
	If~$\bs{\pi}_r^\top\bs{\alpha}<\frac{2}{nl}$, according to Lemma~\ref{centr_d},
	\begin{align}
	s_1\leq\lambda\|Y_\infty\mb{x}_k-\mb{1}_n \otimes \mb{x}^*\|_2,
	\end{align}
	where~$\lambda=\max\left(\left|1-n\bs{\pi}_r^\top\bs{\alpha}\mu\right|,\left|1-n\bs{\pi}_r^\top\bs{\alpha}l \right|\right)$.
	Next we derive a bound for~$s_2$.
	\begin{align}\label{22}
	s_2 =&~(\bs{\pi}_r^\top\bs{\alpha})\left\|n(\mb{1}_n \otimes I_p)\nabla F\big((\bs{\pi}_r^\top\otimes I_p)\mb{x}_k\big)
	-Y_{\infty}\mb{z}_k\right\|_2, \nonumber\\
	\leq
	&~\overline{\alpha}\left\|n(\mb{1}_n \otimes I_p)\nabla F\big((\bs{\pi}_r^\top\otimes I_p)\mb{x}_k\big)
	-(\mb{1}_n\otimes I_p)(\mb{1}_n^\top\otimes I_p)\nabla\mb{f}(\mb{x}_k)\right\|_2, \nonumber\\
	&+~\overline{\alpha}\left\|(\mb{1}_n\otimes I_p)(\mb{1}_n^\top\otimes I_p)\nabla\mb{f}(\mb{x}_k)- Y_{\infty}\mb{z}_k\right\|_2\nonumber\\
	:=&~s_3+s_4, 
	\end{align}
	where it is straightforward to bound~$s_3$ as
	\begin{equation}\label{23}
	s_3 \leq \overline{\alpha} \begin{color}{black}n\end{color}l c\|\mb{x}_k-Y_\infty\mb{x}_k\|.
	\end{equation}
	Since~$Y_{\infty}\widetilde{Y}_{\infty}^{-1}=\left(\mb{1}_n\mb{1}_n^\top\right)\otimes I_p$ and~$Y_{\infty}\mb{z}_k=Y_\infty\widetilde{Y}_k^{-1}\nabla\mb{f}(\mb{x}_k)$ from Lemma~\ref{z}, we have:
	\begin{align}\label{24}
	s_4 =   \overline{\alpha}\left\|Y_{\infty}\widetilde{Y}_{\infty}^{-1}\nabla\mb{f}(\mb{x}_k)-Y_\infty\widetilde{Y}_k^{-1}\nabla\mb{f}(\mb{x}_k)\right\|_2 
	\leq\overline{\alpha}\sqrt{n}ry\widetilde{y}^2\sigma^{k}\left\|\nabla\mb{f}(\mb{x}_k)\right\|_2,
	\end{align}
	where we use Lemma~\ref{yy-}. 
	Combining Eqs.~\eqref{21}-\eqref{24}, we finish the proof.
\end{proof}
Next, we bound~$\|\mb{z}_{k+1}-Y_\infty\mb{z}_{k+1}\|$, the error in gradient estimation.
\begin{lem} \label{3p}
	The following inequality holds,~$\forall k$:
	\begin{align*}
	\|\mb{z}&_{k+1}-Y_\infty\mb{z}_{k+1}\|\nonumber\\ 
	\leq&~ \begin{color}{black}\epsilon\widetilde{y}l\tau cd\end{color}\|\mb{x}_k-Y_\infty\mb{x}_k\|+
	\sigma\|\mb{z}_k-Y_\infty\mb{z}_k\|+\ol{\alpha}\begin{color}{black}\epsilon \widetilde{y} l d\end{color}\|\mb{z}_k\|_2 \nonumber\\
	&+2d\sqrt{n}r\epsilon \widetilde{y}^2\sigma^{k}\|\nabla\mb{f}(\mb{x}_k)\|_2.
	\end{align*}
\end{lem}
\begin{proof}
	According to Eq.~\eqref{alg1_md} and Lemma~\ref{row_ctrac}, we have
	\begin{align}\label{31}
	\|\mb{z}&_{k+1}-Y_{\infty}\mb{z}_{k+1}\| \nonumber\\
	\leq&~\sigma\left\|\mb{z}_k-Y_{\infty}\mb{z}_k\right\|+\left\|\left(\widetilde{Y}_{k+1}^{-1}\nabla \mb{f}(\mb{x}_k)-\widetilde{Y}_k^{-1}\nabla \mb{f}(\mb{x}_k)\right)-\left(Y_{\infty}\mb{z}_{k+1}-Y_{\infty}\mb{z}_k\right)\right\|.
	\end{align}
	Note that~$Y_{\infty}\mb{z}_k=Y_\infty\widetilde{Y}_k^{-1}\nabla\mb{f}(\mb{x}_k)$ from Lemma~\ref{z}. Therefore,
	\begin{align}\label{32}
	&~\left\|\left(\widetilde{Y}_{k+1}^{-1}\nabla \mb{f}(\mb{x}_k)-\widetilde{Y}_k^{-1}\nabla \mb{f}(\mb{x}_k)\right)-\left(Y_{\infty}\mb{z}_{k+1}-Y_{\infty}\mb{z}_k\right)\right\|_2\nonumber\\
	=&~\left\|\left(I_{np}-Y_\infty\right)\left(\widetilde{Y}_{k+1}^{-1}\nabla \mb{f}(\mb{x}_k)-\widetilde{Y}_k^{-1}\nabla \mb{f}(\mb{x}_k)\right)\right\|_2,\nonumber\\
	\leq&~\epsilon\left\|\widetilde{Y}_{k+1}^{-1}\nabla \mb{f}(\mb{x}_k)-\widetilde{Y}_{k+1}^{-1}\nabla\mb{f}(\mb{x}_k)\right\|_2 +\epsilon\left\|\widetilde{Y}_{k+1}^{-1}\nabla\mb{f}(\mb{x}_k)-\widetilde{Y}_k^{-1}\nabla\mb{f}(\mb{x}_k)\right\|_2,\nonumber\\
	\leq&~\epsilon\widetilde{y}l\left\|\mb{x}_{k+1}-\mb{x}_k\right\|_2+2\sqrt{n}r\epsilon \widetilde{y}^2\sigma^{k}\left\|\nabla\mb{f}(\mb{x}_k)\right\|_2,
	\end{align}
	where in the last inequality we use Lemma~\ref{yy-}.
	We now bound~$\|\mb{x}_{k+1}-\mb{x}_k\|_2$. 
	\begin{align}\label{33}
	\left\|\mb{x}_{k+1}-\mb{x}_k\right\|_2\leq&\left\|(A-I_{np})\mb{x}_k\right\|_2+\ol{\alpha}\left\|\mb{z}_k\right\|_2,\nonumber\\
	\leq&\left\|(A-I_{np})\left(\mb{x}_k-Y_{\infty}\mb{x}_k\right)\right\|_2+\ol{\alpha}\left\|\mb{z}_k\right\|_2,\nonumber\\
	\leq&\tau\left\|\mb{x}_k-Y_{\infty}\mb{x}_k\right\|_2+\ol{\alpha}\left\|\mb{z}_k\right\|_2,
	\end{align}
	where in the second inequality we use the fact that~$(A-I_{np})Y_{\infty}$ is a zero matrix.
	Combining Eqs.~\eqref{31}-\eqref{33}, we obtain the desired result.
\end{proof}
The last step is to bound~$\|\mb{z}_k\|_2$ in terms of~$\|\mb{x}_k-Y_\infty\mb{x}_k\|$,~$\|Y_\infty\mb{x}_k-\mb{1}_n \otimes \mb{x}^*\|_2$, and~$\|\mb{z}_k-Y_\infty\mb{z}_k\|$. Then we can replace~$\|\mb{z}_k\|_2$ in Lemma~\ref{1p} and~\ref{3p} by this bound in order to develop a LTI system inequality.
\begin{lem}\label{4p}
	The following inequality holds,~$\forall k$:
	\begin{align}
	\|\mb{z}_k\|_2 \leq&~ cnl\|\mb{x}_k-Y_{\infty}\mb{x}_k\|+nl\|Y_{\infty}\mb{x}_k-\mb{1}_n\otimes\mb{x}^*\|_2\nonumber\\
	&+c\|\mb{z}_k-Y_{\infty}\mb{z}_k\|+\sqrt{n}ry\widetilde{y}^2\sigma^{k}\|\nabla\mb{f}(\mb{x}_k)\|_2.
	\end{align}
\end{lem}
\begin{proof}
	Recall that~$Y_{\infty}\widetilde{Y}_{\infty}^{-1}=(\mb{1}_n\otimes I_p)(\mb{1}_n^\top\otimes I_p)$
	and~$Y_{\infty}\mb{z}_k=Y_\infty\widetilde{Y}_k^{-1}\nabla\mb{f}(\mb{x}_k)$ from Lemma~\ref{z}. We have the following:
	\begin{align}
	\|\mb{z}_k\|_2\leq&~\left\|\mb{z}_k-Y_{\infty}\mb{z}_k\|_2+\|Y_{\infty}\mb{z}_k\right\|_2 \nonumber\\
	\leq&~ c\|\mb{z}_k-Y_{\infty}\mb{z}_k\| + \|Y_\infty\widetilde{Y}_k^{-1}\nabla\mb{f}(\mb{x}_k)-Y_\infty\widetilde{Y}_{\infty}^{-1}\nabla\mb{f}(\mb{x}_k)\|_2
	\nonumber\\
	&~+\|Y_\infty\widetilde{Y}_{\infty}^{-1}\nabla\mb{f}(\mb{x}_k)-(\mb{1}_n\otimes I_p)(\mb{1}_n^\top\otimes I_p)\nabla\mb{f}(\mb{x}^*)\|_2, \nonumber\\
	\leq&~c\|\mb{z}_k-Y_{\infty}\mb{z}_k\|+\sqrt{n}ly\widetilde{y}^2\sigma^{k}\|\nabla\mb{f}(\mb{x}_k)\|_2		\nonumber\\
	&~+nl\|\mb{x}_k-\mb{1}_n\otimes\mb{x}^*\|_2,		\nonumber\\
	\leq&~
	cnl\|\mb{x}_k-Y_{\infty}\mb{x}_k\|+nl\|Y_{\infty}\mb{x}_k-\mb{1}_n\otimes\mb{x}^*\|_2\nonumber\\
	&~+c\|\mb{z}_k-Y_{\infty}\mb{z}_k\|+\sqrt{n}ry\widetilde{y}^2\sigma^{k}\|\nabla\mb{f}(\mb{x}_k)\|_2,
	\end{align}
	where in the second inequality we use the fact that~$(\mb{1}_n^\top\otimes I_p)\nabla\mb{f}(\mb{x}^*)=0$, which is the optimality condition for Problem P1.
\end{proof}
Before the main result, we present an additional lemma from nonnegative matrix theory that will be helpful in establishing the linear convergence of FROST.
\begin{lem}\label{rho1}(Theorem 8.1.29 in~\cite{matrix})
	Let $X\in\mathbb{R}^{n\times n}$ be a nonnegative matrix and~$\mb{x}\in\mathbb{R}^{n}$ be a positive vector. If~$X\mb{x}<\omega\mb{x}$, then~$\rho(X)<\omega$. 
\end{lem}
\subsection{Main results}
With the help of the auxiliary relationships developed in the previous subsection, we now present the main results as follows in Theorems~\ref{thm0} and~\ref{thm1}. Theorem~\ref{thm0} states that the relationships derived in the previous subsection indeed provide a contraction when the largest step-size,~$\overline{\alpha}$, is sufficiently small. Theorem~\ref{thm1} then establishes the linear convergence of FROST.
\begin{theorem}\label{thm0}
	If~$\bs{\pi}_r^\top\bs{\alpha}<\frac{2}{nl}$, the following LTI system inequality holds:
	\begin{equation}
	\mb{t}_{k+1} \leq J_{\bs{\alpha}}\mb{t}_k+H_k\mb{s}_k,~\forall k,\label{JtHs}
	\end{equation}
	where $\mb{t}_k,\mb{s}_k\in\mathbb{R}^3$ and $J_{\bs{\alpha}},H_k\in\mathbb{R}^{3\times3}$ are defined as follows: 
	\begin{align*}
	\mb{t}_k&=\left[
	\begin{array}{l}
	\left\|\mb{x}_k-Y_\infty\mb{x}_k\right\| \\
	\left\|Y_\infty\mb{x}_k-\mb{1}_n \otimes \mb{x}^*\right\|_2 \\
	\left\|\mb{z}_k-Y_\infty\mb{z}_k\right\|
	\end{array}
	\right],\qquad
	J_{\bs{\alpha}}=\left[
	\begin{array}{ccc}
	\sigma+a_1\overline{\alpha} & a_2\overline{\alpha} &a_3\overline{\alpha}\\
	a_4\overline{\alpha} & \lambda & a_5\overline{\alpha}\\
	a_6+a_7\overline{\alpha}& a_8\overline{\alpha} & \sigma+a_9\overline{\alpha}
	\end{array}
	\right],\\
	H_k&=\left[
	\begin{array}{ccc}
	\overline{\alpha} d\epsilon\sqrt{n}ry\widetilde{y}^2 & 0 & 0\\
	\overline{\alpha}\sqrt{n}ry\widetilde{y}^2 & 0 & 0\\
	d\sqrt{n}r\epsilon\widetilde{y}^2\left(2+\overline{\alpha} r y\widetilde{y}\right) & 0 & 0
	\end{array}
	\right]\sigma^k,
	\qquad\mb{s}_k=\left[
	\begin{array}{cc}
	\left\|\nabla\mb{f}(\mb{x}_k)\right\|_2\\
	0\\
	0\\
	\end{array}
	\right],
	\end{align*}
	and the constants $a_i$'s are
	\begin{eqnarray*}
		\begin{array}{lll}
			a_1 = cd\epsilon nl, & 	\qquad a_4 = cnl  & \qquad a_7 = cdnl^2\epsilon\widetilde{y}\\
			a_2 = d\epsilon nl, & \qquad a_5 = yc, &\qquad a_8 = dnl^2\epsilon\widetilde{y}\\
			a_3 = d^2\epsilon, & \qquad
			a_6 = \epsilon\widetilde{y}l\tau cd,& \qquad
			a_{9} = d^2\epsilon l\widetilde{y}.
		\end{array}
	\end{eqnarray*}
	Let~$[\bs{\pi}_r]_-$ be the smallest element in~$\bs{\pi}_r$. When the largest step-size,~$\ol{\alpha}$, satisfies
	\begin{align}\label{eb1}
	0<\ol{\alpha}<\min\left\{~\frac{\delta_1(1-\sigma)}{a_1\delta_1+a_2\delta_2+a_3\delta_3},~\frac{(1-\sigma)\delta_3-\delta_1a_6}{a_7\delta_1+a_8\delta_2+a_9\delta_3},~\frac{1}{nl}\right\},
	\end{align} 
	with positive constants~$\delta_1,\delta_2,\delta_3$ such that
	\begin{align}\label{eb2}
	\delta_3> 0,\qquad\delta_1 < \frac{(1-\sigma)\delta_3}{a_6}, \qquad 
	\delta_2 > \frac{a_4\delta_1+a_5\delta_3}{\mu n[\bs{\pi}_r]_-}, 
	\end{align}
	then the spectral radius of $J_{\bs{\alpha}}$ is strictly less than~1.
\end{theorem}

\begin{proof}
	Combining Lemmas~\ref{1p}--\ref{4p}, one can verify that Eq.~\eqref{JtHs} holds if~$\bs{\pi}_r^\top\bs{\alpha}<\frac{2}{nl}$. Recall that~$\lambda=\max\left(\left|1-\mu n\bs{\pi}_r^\top\bs{\alpha}\right|,\left|1-l n\bs{\pi}_r^\top\bs{\alpha} \right|\right)$. When~$\bs{\pi}_r^\top\bs{\alpha}<\frac{1}{nl}$,~$\lambda=1-\mu n\bs{\pi}_r^\top\bs{\alpha}$, since $\mu\leq l$~\cite{nesterov2013introductory}. In order to make~$\bs{\pi}_r^\top\bs{\alpha}<\frac{1}{nl}$ hold, it is suffice to require~$\ol{\alpha}<\frac{1}{nl}$. The next step is to find an upper bound,~$\hat{\alpha}$, on the largest step-size such that~$\rho(J_{\bs{\alpha}})<1$ when~$\ol{\alpha}<\hat{\alpha}.$ In the light of Lemma~\ref{rho1}, we solve for the range of the largest step-size,~$\ol{\alpha}$, and a positive vector~$\bs{\delta}=\left[\delta_1,\delta_2,\delta_3\right]^\top$ from the following:
	\begin{align}\label{eta1}
	\left[
	\begin{array}{ccc}
	\sigma+a_1\ol{\alpha} & a_2\ol{\alpha} &a_3\ol{\alpha} \\
	a_4\ol{\alpha} & 1-\mu n(\bs{\pi}_r^\top\bs{\alpha}) & a_5\ol{\alpha}\\
	a_6+a_7\ol{\alpha}& a_8\ol{\alpha} & \sigma+a_{9}\ol{\alpha}
	\end{array}
	\right]
	\left[
	\begin{array}{ccc}
	\delta_1 \\
	\delta_2\\
	\delta_3
	\end{array}
	\right]
	<
	\left[
	\begin{array}{ccc}
	\delta_1 \\
	\delta_2\\
	\delta_3
	\end{array}
	\right],
	\end{align}
	which is equivalent to the following set of inequalities: 
	\begin{align}\label{delt}
	\left\{
	\begin{array}{rll}
	(a_1\delta_1+a_2\delta_2+a_3\delta_3)\ol{\alpha}&<&\delta_1(1-\sigma), \\
	(a_4\delta_1+a_5\delta_3)\ol{\alpha}-\delta_2\mu n\bs{\pi}_r^\top\bs{\alpha}&<&0, \\
	(a_7\delta_1+a_8\delta_2+a_9\delta_3)\ol{\alpha}&<&(1-\sigma)\delta_3-\delta_1a_6.
	\end{array} 
	\right.
	\end{align}
	Since the right hand side of the third inequality in Eqs.~\eqref{delt} has to be positive, we have that: 
	\begin{equation}\label{e1}
	0<\delta_1 < \frac{(1-\sigma)\delta_3}{a_6}.
	\end{equation}
	In order to find the range of~$\delta_2$ such that the second inequality holds, it suffices to solve for the range of~$\delta_2$ such that the following inequality holds:
	\begin{align*}
	(a_4\delta_1+a_5\delta_3)\ol{\alpha}-\delta_2\mu n[\bs{\pi}_r]_-\ol{\alpha}<0,
	\end{align*}
	where~$[\bs{\pi}_r]_-$ is the smallest entry in~$\bs{\pi}_r$.
	Therefore, as long as 
	\begin{equation}\label{e2}
	\delta_2 > \frac{a_4\delta_1+a_5\delta_3}{\mu n[\bs{\pi}_r]_-},
	\end{equation}
	the second inequality in Eqs.~\eqref{delt} holds. The next step is to solve the range of~$\ol{\alpha}$ from the first and third inequalities in Eqs.~\eqref{delt}. We get
	\begin{align*}
	\ol{\alpha} < \min\left\{~\frac{\delta_1(1-\sigma)}{a_1\delta_1+a_2\delta_2+a_3\delta_3},~\frac{(1-\sigma)\delta_3-\delta_1a_6}{a_7\delta_1+a_8\delta_2+a_9\delta_3}~\right\},
	\end{align*}
	where the range of~$\delta_1$ and~$\delta_2$ is given in~Eq.~\eqref{e1} and Eq.~\eqref{e2}, respectively, and~$\delta_3$ is an arbitrary positive constant and the theorem follows. 
\end{proof}
{\color{black}Note that~$\delta_1,\delta_2,\delta_3$ are essentially adjustable parameters that are chosen independently from the step-sizes. Specifically, according to Eq.~\eqref{eb2}, we first choose an arbitrary positive constant~$\delta_3$ and subsequently choose a constant~$\delta_1$ such that~$0< \delta_1 < \frac{(1-\sigma)\delta_3}{a_6}$ and finally we choose a constant~$\delta_2$ such that $\delta_2>\frac{a_4\delta_1+a_5\delta_3}{\mu n[\bs{\pi}_r]_-}$.}
\begin{theorem}\label{thm1}
	If the largest step-size~$\ol{\alpha}$ follows the bound in Eq.~\eqref{eb1}, we have:
	\begin{equation*}
	\left\|\mb{x}_k-\mb{1}_n\otimes \mb{x}^*\right\|\leq m\big(\max\left\{\rho\left(J_{\bs{\alpha}}\right),\sigma\right\}+\xi\big)^k,
	\end{equation*}
	where~$\xi$ is an arbitrarily small constant,~$\sigma$ is the contraction factor defined in Lemma~\ref{row_ctrac} and~$m$ is some positive constant.
\end{theorem}
Noticing that~$\rho\left(J_{\bs{\alpha}}\right)<1$ when the largest step-size,~$\ol{\alpha}$, follows the bound in Eq.~\eqref{eb1} and that~$H_k$ linearly decays at the rate of~$\sigma^k$, one can intuitively verify Theorem~\ref{thm1}. A rigorous proof follows from~\cite{linear_row}.

In Theorems~\ref{thm0} and~\ref{thm1}, we establish the linear convergence of FROST when the largest step-size,~$\ol{\alpha}$, follows the upper bound defined in Eqs.~\eqref{eb1}. {\color{black}Distributed optimization (based on gradient tracking)} with uncoordinated step-sizes have been previously studied in~\cite{xu2015augmented,xu2018convergence,nedic2017geometrically}, over undirected graphs with doubly-stochastic weights, and in~\cite{lu2018geometrical}, over directed graphs with column-stochastic weights. These works rely on some notion of heterogeneity of the step-sizes, defined respectively as the relative deviation of the step-sizes from their average,~$\frac{\|(I_n-U)\bs{\alpha}\|_2}{\|U\bs{\alpha}\|_2}$, where~$U=\mb{1}_n\mb{1}_n^\top/n$, in~\cite{xu2015augmented,xu2018convergence}, and as the ratio of the largest to the smallest step-size,~$\frac{\max_i\{\alpha_i\}}{\min_i\{\alpha_i\}}$, in~\cite{nedic2017geometrically,lu2018geometrical}. The authors then show that when the heterogeneity is small enough and when the largest step-size follows a bound that is a function of the heterogeneity, the proposed algorithms converge to the optimal solution. It is worth noting that sufficiently small step-sizes cannot guarantee sufficiently small heterogeneity in both of the aforementioned definitions. In contrast, the upper bound on the largest step-size in this paper, Eqs.~\eqref{eb1} and~\eqref{eb2}, is independent of any notion of heterogeneity and only depends on the objective functions and the network parameters\footnote{The constants~$\delta_1$,~$\delta_2$ and~$\delta_3$ in Eqs.~\eqref{eb1} and~\eqref{eb2} are tunable parameters that only depend on the network topology and objective functions.}. Each agent therefore locally picks a sufficiently small step-size independent of other step-sizes. Besides, this bound allows the agents to choose a zero step-size as long as at least one of them is positive and sufficiently small. 

\section{Numerical Results}\label{s6}
In this section, we use numerical experiments to support the theoretical results. We consider a distributed logistic regression problem. Each agent~$i$ has access to~$m_i$ training data,~$(\mb{c}_{ij},y_{ij})\in\mathbb{R}^p\times\{-1,+1\}$, where~$\mb{c}_{ij}$ contains~$p$ features of the~$j$th training data at agent~$i$ and~$y_{ij}$ is the corresponding binary label. The network of agents cooperatively solves the following distributed logistic regression problem:
\begin{align}
\underset{\mb{w}\in\mbb{R}^p,b\in\mathbb{R}}{\operatorname{min}}~F(\mb{w},b)=\sum_{i=1}^n\sum_{j=1}^{m_i}\ln\left[1+\exp\left(-\left(\mb{w}^\top\mb{c}_{ij}+b\right)y_{ij}\right)\right]+\frac{n\lambda}{2}\|\mb{w}\|_2^2\nonumber,
\end{align}
with each private loss function being
\begin{equation}
f_i(\mb{w},b)=\sum_{j=1}^{m_i}\ln\left[1+\exp\left(-\left(\mb{w}^\top\mb{c}_{ij}+b\right)y_{ij}\right)\right]+\frac{\lambda}{2}\|\mb{w}\|_2^2,
\end{equation}
where~$\frac{\lambda}{2}\|\mb{w}\|_2^2$ is a regularization term used to prevent overfitting of the data. The feature vectors,~$\mb{c}_{ij}$'s, are  randomly generated from some Gaussian distribution with zero mean. The binary labels are randomly generated from some Bernoulli distribution. The network topology is shown in Fig~\ref{graph}. 
	\begin{figure}[!h]
		\centering
		\subfigure{\includegraphics[width=2.7in]{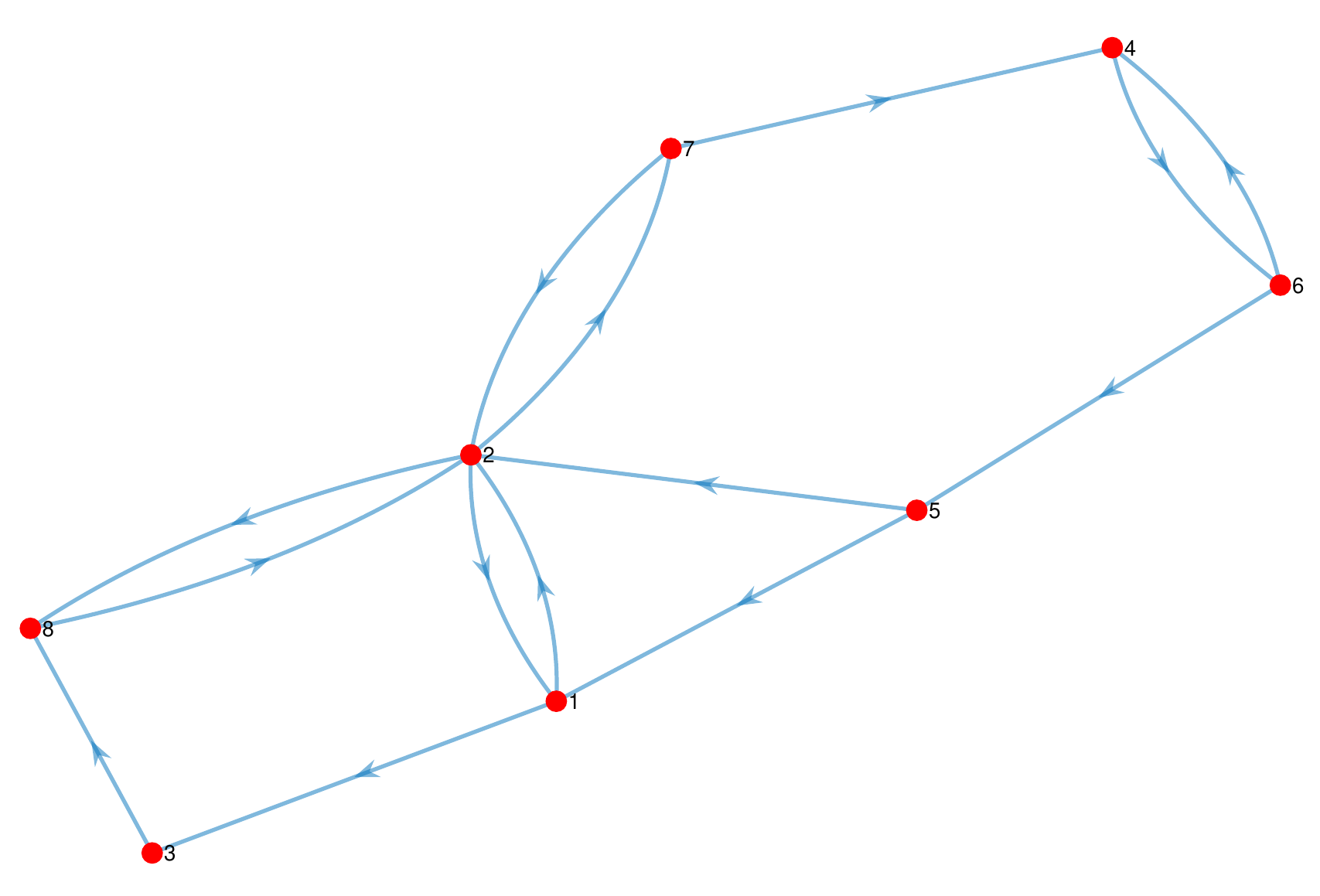}}
		\caption{A strongly-connected and unbalanced directed graph.}
		\label{graph}
	\end{figure}
	
We adopt a simple uniform weighting strategy to construct the row- and column-stochastic weights when needed:~$a_{ij}=1/|\mc{N}_i^{{\scriptsize \mbox{in}}}|,~b_{ij}=1/|\mc{N}_j^{{\scriptsize \mbox{out}}}|,~\forall i,j$. We plot the average of residuals at each agent,~$\frac{1}{n}\sum_{i=1}^{n}\|\mb{x}_i(k)-\mb{x}^*\|_2$. In Fig.~\ref{conv} (left), each curve represents the linear convergence of FROST when the corresponding agent uses a positive step-size, optimized manually, while every other agent uses zero step-size. In Fig.~\ref{conv} (right), we compare the performance of FROST, with ADD-OPT/Push-DIGing~\cite{xi2017add,diging}, see Section~\ref{saddopt}, and with the~$\mc{AB}$ algorithm in~\cite{AB,ABM}, see Section~\ref{slinear}. The step-size used in each algorithm is optimized. For FROST, we first manually find the optimal identical step-size for all agents, which is~$0.07$ in our experiment, and then randomly generate uncoordinated step-sizes of FROST from the uniform distribution over the interval~$[0,0.07]$ (therefore, the convergence speed of FROST shown in this experiment is conservative). The numerical experiments thus verify our theoretical finding that as long as the largest step-size of FROST is positive and sufficiently small, FROST linearly converges to the optimal solution.
	\begin{figure}[!h]
		\centering
		\subfigure{\includegraphics[width=6in]{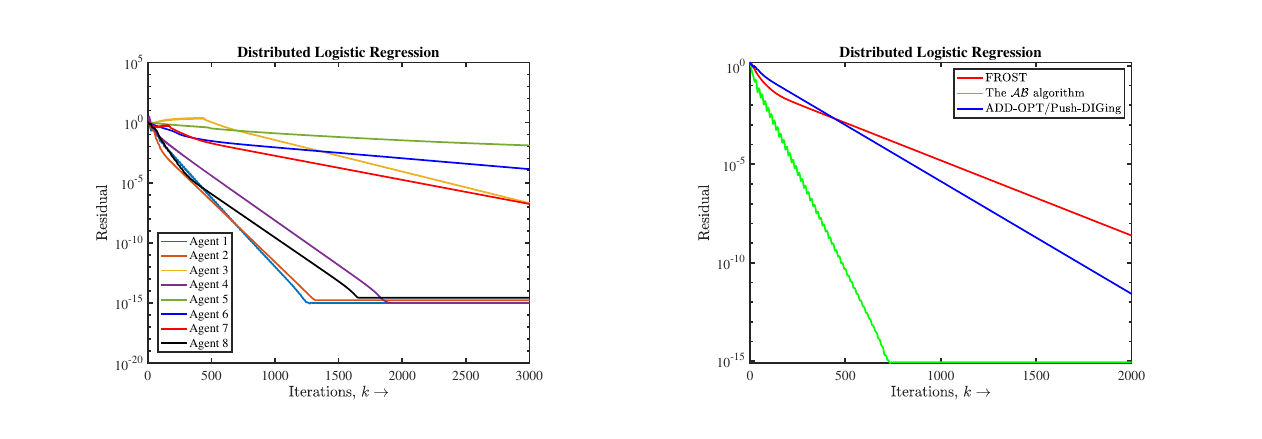}}
		\caption{(Left) Each curve represents the linear convergence of FROST when the corresponding agent uses a positive step-size, optimized manually, while every other agent uses zero step-size. (Right) Convergence comparison across different algorithms.}
		\label{conv}
	\end{figure}

In the next experiment, we show the influence of the network sparsity on the convergence of FROST. For this purpose, we use three different graphs each with~$n=50$ nodes, where~$\mc{G}_1$ has roughly~$10\%$ of total edges,~$\mc{G}_2$ has roughly~$13\%$ of total edges, and~$\mc{G}_3$ has roughly~$16\%$ of total edges. These graphs are shown in Fig.~\ref{G123} and the performance of FROST over each one of them is shown in Fig.~\ref{pG}. 
	
		\begin{figure}[!h]
			\centering
			\subfigure{\includegraphics[width=6.5in]{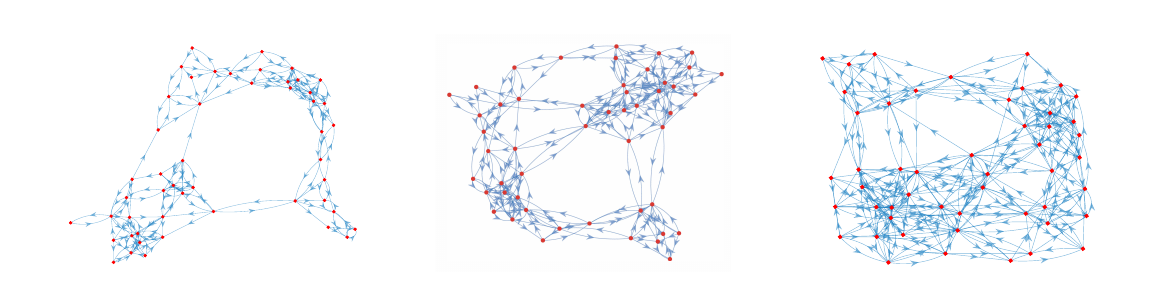}}
			\caption{Directed graphs with~$n=50$ nodes and increasing sparsity:~$\mc{G}_1,\mc{G}_2$, and~$\mc{G}_3$.}
			\label{G123}
		\end{figure}
		
		\begin{figure}[!h]
			\centering
			\includegraphics[width=2.7in]{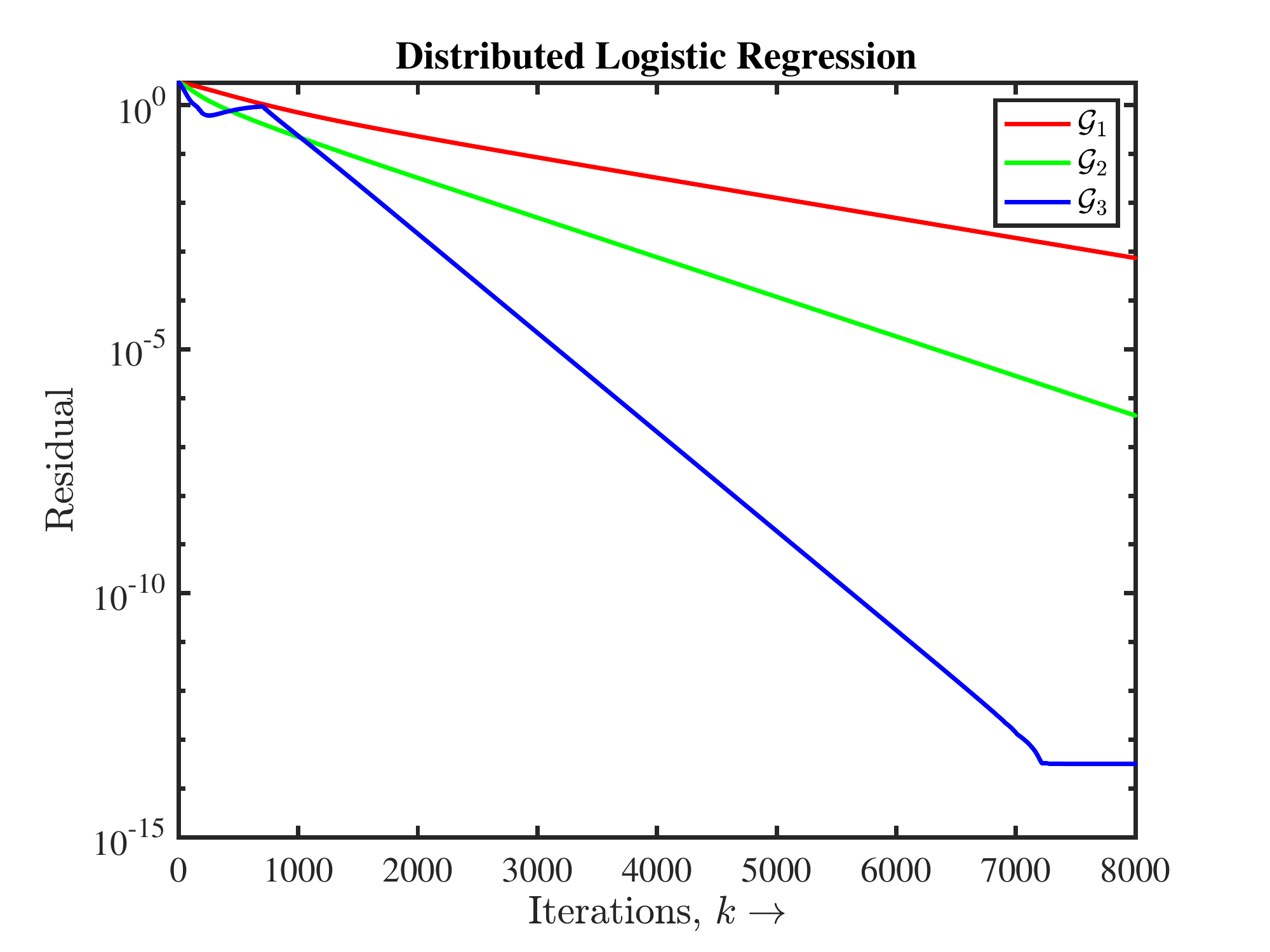}
			\caption{Influence of network sparsity on the performance of FROST.}
			\label{pG}
		\end{figure}

\section{Conclusions}\label{s7}
In this paper, we consider distributed optimization applicable to both directed and undirected graphs with row-stochastic weights and when the agents in the network have uncoordinated step-sizes. Most of the existing algorithms are based on column-stochastic weights, which may be infeasible to implement in many practical scenarios. Row-stochastic weights, on the other hand, are straightforward to implement as each agent locally determines the weights assigned to each incoming information. We propose a fast algorithm that we call FROST (Fast Row-stochastic Optimization with uncoordinated STep-sizes) and show that when the largest step-size is positive and sufficiently small, FROST linearly converges to the optimal solution. Simulation results further verify the theoretical analysis.

	\bibliographystyle{IEEEbib}
	\bibliography{sample,sample_2,ref,sample_ABM}

\begin{thebibliography}{10}

\bibitem{forero2010consensus}
P.~A. Forero, A.~Cano, and G.~B. Giannakis,
\newblock ``Consensus-based distributed support vector machines,''
\newblock {\em Journal of Machine Learning Research}, vol. 11, no. May, pp.
  1663--1707, 2010.

\bibitem{distributed_Boyd}
S.~Boyd, N.~Parikh, E.~Chu, B.~Peleato, and J.~Eckstein,
\newblock ``Distributed optimization and statistical learning via the
  alternating direction method of multipliers,''
\newblock {\em Foundation and Trends in Maching Learning}, vol. 3, no. 1, pp.
  1--122, Jan. 2011.

\bibitem{raja2016cloud}
``{Cloud K-SVD: A collaborative dictionary learning algorithm for big,
  distributed data}, author={Raja, H. and Bajwa, W. U.}, journal={IEEE Trans.
  on Signal Processing}, volume={64}, number={1}, pages={173--188},
  year={2016}, publisher={IEEE},''
\newblock .

\bibitem{wai2018multi}
H.-T. Wai, Z.~Yang, Z.~Wang, and M.~Hong,
\newblock ``Multi-agent reinforcement learning via double averaging primal-dual
  optimization,''
\newblock {\em arXiv preprint arXiv:1806.00877}, 2018.

\bibitem{di2013sparse}
Paolo Di~Lorenzo and Ali~H Sayed,
\newblock ``Sparse distributed learning based on diffusion adaptation,''
\newblock {\em IEEE Transactions on signal processing}, vol. 61, no. 6, pp.
  1419--1433, 2013.

\bibitem{jadbabaie2003coordination}
A.~Jadbabaie, J.~Lin, and A.~Morse,
\newblock ``Coordination of groups of mobile autonomous agents using nearest
  neighbor rules,''
\newblock {\em IEEE Trans. on Automatic Control}, vol. 48, no. 6, pp.
  988--1001, 2003.

\bibitem{distributed_Mateos}
G.~Mateos, J.~A. Bazerque, and G.~B. Giannakis,
\newblock ``Distributed sparse linear regression,''
\newblock {\em IEEE Trans. on Signal Processing}, vol. 58, no. 10, pp.
  5262--5276, Oct. 2010.

\bibitem{distributed_Bazerque}
J.~A. Bazerque and G.~B. Giannakis,
\newblock ``Distributed spectrum sensing for cognitive radio networks by
  exploiting sparsity,''
\newblock {\em IEEE Trans. on Signal Processing}, vol. 58, no. 3, pp.
  1847--1862, March 2010.

\bibitem{distributed_Rabbit}
M.~Rabbat and R.~Nowak,
\newblock ``Distributed optimization in sensor networks,''
\newblock in {\em 3rd International Symposium on Information Processing in
  Sensor Networks}, Berkeley, CA, Apr. 2004, pp. 20--27.

\bibitem{safavi2018distributed}
S.~Safavi, U.~A. Khan, S.~Kar, and J.~M.~F. Moura,
\newblock ``Distributed localization: A linear theory,''
\newblock {\em Proceedings of the IEEE}, vol. 106, no. 7, pp. 1204--1223, July
  2018.

\bibitem{DOPT1}
J.~Tsitsiklis, D.~P. Bertsekas, and M.~Athans,
\newblock ``Distributed asynchronous deterministic and stochastic gradient
  optimization algorithms,''
\newblock {\em IEEE Transactions on Automatic Control}, vol. 31, no. 9, pp.
  803--812, 1986.

\bibitem{uc_Nedic}
A.~Nedi\'{c} and A.~Ozdaglar,
\newblock ``Distributed subgradient methods for multi-agent optimization,''
\newblock {\em IEEE Trans. on Automatic Control}, vol. 54, no. 1, pp. 48--61,
  Jan. 2009.

\bibitem{DGD_Yuan}
K.~Yuan, Q.~Ling, and W.~Yin,
\newblock ``On the convergence of decentralized gradient descent,''
\newblock {\em SIAM Journal on Optimization}, vol. 26, no. 3, pp. 1835--1854,
  Sep. 2016.

\bibitem{balancing}
A.~S. Berahas, R.~Bollapragada, N.~S. Keskar, and E.~Wei,
\newblock ``Balancing communication and computation in distributed
  optimization,''
\newblock {\em arXiv preprint arXiv:1709.02999}, 2017.

\bibitem{dual_Terelius}
H.~Terelius, U.~Topcu, and R.~M. Murray,
\newblock ``Decentralized multi-agent optimization via dual decomposition,''
\newblock {\em IFAC Proceedings Volumes}, vol. 44, no. 1, pp. 11245--11251,
  2011.

\bibitem{ADMM_Mota}
J.~F.~C. Mota, J.~M.~F. Xavier, P.~M.~Q. Aguiar, and M.~Puschel,
\newblock ``D-{ADMM}: A communication-efficient distributed algorithm for
  separable optimization,''
\newblock {\em IEEE Trans. on Signal Processing}, vol. 61, no. 10, pp.
  2718--2723, May 2013.

\bibitem{ADMM_Wei}
E.~Wei and A.~Ozdaglar,
\newblock ``Distributed alternating direction method of multipliers,''
\newblock in {\em 51st IEEE Annual Conference on Decision and Control}, Dec.
  2012, pp. 5445--5450.

\bibitem{ADMM_Shi}
W.~Shi, Q.~Ling, K~Yuan, G~Wu, and W~Yin,
\newblock ``On the linear convergence of the admm in decentralized consensus
  optimization,''
\newblock {\em IEEE Trans. on Signal Processing}, vol. 62, no. 7, pp.
  1750--1761, April 2014.

\bibitem{fast_Gradient}
D.~Jakovetic, J.~Xavier, and J.~M.~F. Moura,
\newblock ``Fast distributed gradient methods,''
\newblock {\em IEEE Trans. on Automatic Control}, vol. 59, no. 5, pp.
  1131--1146, 2014.

\bibitem{EXTRA}
W.~Shi, Q.~Ling, G.~Wu, and W~Yin,
\newblock ``Extra: An exact first-order algorithm for decentralized consensus
  optimization,''
\newblock {\em SIAM Journal on Optimization}, vol. 25, no. 2, pp. 944--966,
  2015.

\bibitem{exactdiffusion1}
K.~Yuan, B.~Ying, X.~Zhao, and A.~H. Sayed,
\newblock ``Exact diffusion for distributed optimization and learning - part
  {I}: Algorithm development,''
\newblock {\em IEEE Transactions on Signal Processing}, pp. 1--1, 2018.

\bibitem{exactdiffusion2}
K.~Yuan, B.~Ying, X.~Zhao, and A.~H. Sayed,
\newblock ``Exact diffusion for distributed optimization and learning - part
  {II}: Convergence analysis,''
\newblock {\em IEEE Transactions on Signal Processing}, pp. 1--1, 2018.

\bibitem{diffusion}
A.~H. Sayed,
\newblock ``Diffusion adaptation over networks,''
\newblock in {\em Academic Press Library in Signal Processing}, vol.~3, pp.
  323--453. Elsevier, 2014.

\bibitem{DAC}
M.~Zhu and S.~Mart{\'\i}nez,
\newblock ``Discrete-time dynamic average consensus,''
\newblock {\em Automatica}, vol. 46, no. 2, pp. 322--329, 2010.

\bibitem{AugDGM}
J.~Xu, S.~Zhu, Y.~C. Soh, and L.~Xie,
\newblock ``Augmented distributed gradient methods for multi-agent optimization
  under uncoordinated constant stepsizes,''
\newblock in {\em IEEE 54th Annual Conference on Decision and Control}, 2015,
  pp. 2055--2060.

\bibitem{preNEXT}
P.~Di Lorenzo and G.~Scutari,
\newblock ``Distributed nonconvex optimization over time-varying networks,''
\newblock in {\em 2016 IEEE International Conference on Acoustics, Speech and
  Signal Processing (ICASSP)}, March 2016, pp. 4124--4128.

\bibitem{harness}
G.~Qu and N.~Li,
\newblock ``Harnessing smoothness to accelerate distributed optimization,''
\newblock {\em IEEE Transactions on Control of Network Systems}, vol. 5, no. 3,
  pp. 1245--1260, Sept 2018.

\bibitem{NEXT}
P.~Di~Lorenzo and G.~Scutari,
\newblock ``Next: In-network nonconvex optimization,''
\newblock {\em IEEE Trans. on Signal and Information Processing over Networks},
  vol. 2, no. 2, pp. 120--136, 2016.

\bibitem{opdirect_Tsianous}
K.~I. Tsianos, S.~Lawlor, and M.~G. Rabbat,
\newblock ``Push-sum distributed dual averaging for convex optimization,''
\newblock in {\em 51st IEEE Annual Conference on Decision and Control}, Maui,
  Hawaii, Dec. 2012, pp. 5453--5458.

\bibitem{opdirect_Nedic}
A.~Nedi\'{c} and A.~Olshevsky,
\newblock ``Distributed optimization over time-varying directed graphs,''
\newblock {\em IEEE Trans. on Automatic Control}, vol. 60, no. 3, pp. 601--615,
  Mar. 2015.

\bibitem{D-DGD}
C.~Xi, Q.~Wu, and U.~A. Khan,
\newblock ``On the distributed optimization over directed networks,''
\newblock {\em Neurocomputing}, vol. 267, pp. 508--515, Dec. 2017.

\bibitem{D-DPS}
C.~Xi and U.~A. Khan,
\newblock ``Distributed subgradient projection algorithm over directed
  graphs,''
\newblock {\em IEEE Trans. on Automatic Control}, vol. 62, no. 8, pp.
  3986--3992, Oct. 2016.

\bibitem{AsySP}
Mahmoud Assran and Michael Rabbat,
\newblock ``Asynchronous subgradient-push,''
\newblock {\em arXiv preprint arXiv:1803.08950}, 2018.

\bibitem{AsySPA}
Jiaqi Zhang and Keyou You,
\newblock ``Asyspa: An exact asynchronous algorithm for convex optimization
  over digraphs,''
\newblock {\em arXiv preprint arXiv:1808.04118}, 2018.

\bibitem{RstSP}
Alex Olshevsky, Ioannis~Ch Paschalidis, and Artin Spiridonoff,
\newblock ``Robust asynchronous stochastic gradient-push: Asymptotically
  optimal and network-independent performance for strongly convex functions,''
\newblock {\em arXiv preprint arXiv:1811.03982}, 2018.

\bibitem{DEXTRA}
C.~Xi and U.~A. Khan,
\newblock ``{DEXTRA: A} fast algorithm for optimization over directed graphs,''
\newblock {\em IEEE Trans. on Automatic Control}, vol. 62, no. 10, pp.
  4980--4993, Oct. 2017.

\bibitem{ac_directed0}
D.~Kempe, A.~Dobra, and J.~Gehrke,
\newblock ``Gossip-based computation of aggregate information,''
\newblock in {\em 44th Annual IEEE Symposium on Foundations of Computer
  Science}, Oct. 2003, pp. 482--491.

\bibitem{xi2017add}
C.~Xi, R.~Xin, and U.~A. Khan,
\newblock ``{ADD-OPT: Accelerated distributed directed optimization},''
\newblock {\em IEEE Transactions on Automatic Control}, vol. 63, no. 5, pp.
  1329--1339, 2017.

\bibitem{diging}
A.~Nedi\'{c}, A.~Olshevsky, and W.~Shi,
\newblock ``Achieving geometric convergence for distributed optimization over
  time-varying graphs,''
\newblock {\em SIAM Journal on Optimization}, vol. 27, no. 4, pp. 2597--2633,
  2017.

\bibitem{sonata}
Y.~Sun, G.~Scutari, and D.~Palomar,
\newblock ``Distributed nonconvex multiagent optimization over time-varying
  networks,''
\newblock in {\em 2016 50th Asilomar Conference on Signals, Systems and
  Computers}. IEEE, 2016, pp. 788--794.

\bibitem{AB}
R.~Xin and U.~A. Khan,
\newblock ``A linear algorithm for optimization over directed graphs with
  geometric convergence,''
\newblock {\em IEEE Control Systems Letters}, vol. 2, no. 3, pp. 325--330, Jul.
  2018.

\bibitem{ABM}
Ran Xin and Usman~A Khan,
\newblock ``Distributed heavy-ball: A generalization and acceleration of
  first-order methods with gradient tracking,''
\newblock {\em arXiv preprint arXiv:1808.02942}, 2018.

\bibitem{mai2016distributed}
Van~Sy Mai and E.~H. Abed,
\newblock ``Distributed optimization over weighted directed graphs using row
  stochastic matrix,''
\newblock in {\em 2016 American Control Conference (ACC)}, July 2016, pp.
  7165--7170.

\bibitem{linear_row}
C.~Xi, V.~S. Mai, R.~Xin, E.~Abed, and U.~A. Khan,
\newblock ``Linear convergence in optimization over directed graphs with
  row-stochastic matrices,''
\newblock {\em IEEE Trans. on Automatic Control}, Jan. 2018,
\newblock \textit{in press}.

\bibitem{xu2018convergence}
J.~Xu, S.~Zhu, Y.~Soh, and L.~Xie,
\newblock ``Convergence of asynchronous distributed gradient methods over
  stochastic networks,''
\newblock {\em IEEE Transactions on Automatic Control}, vol. 63, no. 2, pp.
  434--448, 2018.

\bibitem{nedic2017geometrically}
A.~Nedić, A.~Olshevsky, W.~Shi, and C.~A. Uribe,
\newblock ``Geometrically convergent distributed optimization with
  uncoordinated step-sizes,''
\newblock in {\em 2017 American Control Conference (ACC)}, May 2017, pp.
  3950--3955.

\bibitem{lu2018geometrical}
Q.~L{\"u}, H.~Li, and D.~Xia,
\newblock ``Geometrical convergence rate for distributed optimization with
  time-varying directed graphs and uncoordinated step-sizes,''
\newblock {\em Information Sciences}, vol. 422, pp. 516--530, 20018.

\bibitem{xu2015augmented}
J.~Xu, Sj~Zhu, Yj~Cj Soh, and L.~Xie,
\newblock ``Augmented distributed gradient methods for multi-agent optimization
  under uncoordinated constant stepsizes,''
\newblock in {\em IEEE 54th Annual Conference on Decision and Control}, 2015,
  pp. 2055--2060.

\bibitem{matrix}
R.~A. Horn and C.~R. Johnson,
\newblock {\em Matrix Analysis, 2${\mbox{\scriptsize nd}}$ ed.},
\newblock Cambridge University Press, New York, NY, 2013.

\bibitem{kuhn55}
H.~W. Kuhn,
\newblock ``The {H}ungarian method for the assignment problem,''
\newblock {\em Naval Research Logistics Quarterly}, vol. 2, no. 1-2, pp.
  83--97, Mar. 1955.

\bibitem{saf_asil:14}
S.~Safavi and U.~A. Khan,
\newblock ``On the convergence rate of swap-collide algorithm for simple task
  assignment,''
\newblock in {\em 48th IEEE Asilomar Conference on Signals, Systems, and
  Computers}, Pacific Grove, CA, Nov. 2014, pp. 1507--1510.

\bibitem{zhu2010discrete}
M.~Zhu and S.~Mart{\'\i}nez,
\newblock ``Discrete-time dynamic average consensus,''
\newblock {\em Automatica}, vol. 46, no. 2, pp. 322--329, 2010.

\bibitem{ac_directed}
F.~Benezit, V.~Blondel, P.~Thiran, J.~Tsitsiklis, and M.~Vetterli,
\newblock ``Weighted gossip: Distributed averaging using non-doubly stochastic
  matrices,''
\newblock in {\em IEEE International Symposium on Information Theory}, Jun.
  2010, pp. 1753--1757.

\bibitem{tv-ab}
Fakhteh Saadatniaki, Ran Xin, and Usman~A Khan,
\newblock ``Optimization over time-varying directed graphs with row and
  column-stochastic matrices,''
\newblock {\em arXiv preprint arXiv:1810.07393}, 2018.

\bibitem{ac_Cai1}
K.~Cai and H.~Ishii,
\newblock ``Average consensus on general strongly connected digraphs,''
\newblock {\em Automatica}, vol. 48, no. 11, pp. 2750 -- 2761, 2012.

\bibitem{abLS}
Tao Yang, Jemin George, Jiahu Qin, Xinlei Yi, and Junfeng Wu,
\newblock ``Distributed finite-time least squares solver for network linear
  equations,''
\newblock {\em arXiv preprint arXiv:1810.00156}, 2018.

\bibitem{hornjohnson:13}
R.~A. Horn and C.~R. Johnson,
\newblock {\em Matrix Analysis},
\newblock Cambridge University Press, New York, NY, 2013.

\bibitem{bertsekas1999nonlinear}
D.~P. Bertsekas,
\newblock {\em Nonlinear programming},
\newblock Athena scientific Belmont, 1999.

\bibitem{nesterov2013introductory}
Y.~Nesterov,
\newblock {\em Introductory lectures on convex optimization: A basic course},
  vol.~87,
\newblock Springer Science \& Business Media, 2013.

\end{thebibliography}

\end{document}